\documentclass[11pt,a4paper]{amsart}

\usepackage[english]{babel}
\usepackage[utf8x]{inputenc}
\usepackage[T1]{fontenc}
\usepackage{todonotes}
\usepackage{tikz}

\usepackage[a4paper,top=3cm,bottom=3cm,left=3cm,right=3cm,marginparwidth=1.75cm]{geometry}

\usepackage{graphicx}
\usepackage{amsmath,amsthm,amsfonts}
\usepackage{amssymb}
\usepackage{mathtools}
\usepackage{graphicx}
\usepackage{stmaryrd}
\usepackage{scalerel}
\usepackage{mathrsfs}
\usepackage{relsize}
\usepackage{enumerate}   
\usepackage{bm,dsfont} 
\usepackage{commath}

\usepackage{varioref}
\usepackage{hyperref}
\usepackage[nameinlink, capitalize, noabbrev]{cleveref}

\theoremstyle{plain} %% needs `amsmath' package
\newtheorem{thm}{Theorem}[section]
\newtheorem*{theorem*}{Theorem}
\newtheorem{lemma}[thm]{Lemma}
\newtheorem{prop}[thm]{Proposition}
\newtheorem{corl}[thm]{Corollary}

\newtheorem{conv}[thm]{Convention}

\theoremstyle{definition} %% needs `amsmath' package
\newtheorem{defn}[thm]{Definition}
\theoremstyle{remark} %% needs `amsmath' package
\newtheorem{rmk}[thm]{Remark}
\numberwithin{equation}{section}

\DeclareMathOperator{\End}{End}   %% endomorphism algebra
\DeclareMathOperator{\Id}{Id}     %% identity map
\DeclareMathOperator{\linspan}{span} %% linear span
\DeclareMathOperator{\tr}{tr}     %% matrix trace
\DeclareMathOperator{\blangle}{{}_{\bullet}\langle} %%Bullet langle
\DeclareMathOperator{\brangle}{\rangle_{\bullet}} %%Bullet rangle
\DeclareMathOperator{\bracketb}{]_{\bullet}} %%Bullet right bracket
\DeclareMathOperator{\bbracket}{{}_{\bullet}[} %%Bullet left bracket
\newcommand{\A}{\mathcal{A}}		%% Smooth subalgebra
\newcommand{\B}{\mathcal{B}}		%% Smooth subalgebra
\newcommand{\R}{\mathbb{R}}			%% Real numbers
\newcommand{\C}{\mathbb{C}}			%% Complex numbers
\newcommand{\Z}{\mathbb{Z}}			%% Integers
\newcommand{\N}{\mathbb{N}}			%% Natural numbers
\newcommand{\E}{\mathcal{E}}        %% Hilbert module
\newcommand{\K}{\mathbb{K}}

\newcommand{\GS}{\mathcal{G}}		%% Gabor system
\newcommand{\La}{\Lambda}     %% short for \Lambda
\newcommand{\la}{\lambda}     %% short for \lambda
\newcommand{\Lao}{\Lambda^{\circ}}     %% short for \Lambda circ
\newcommand{\lao}{\lambda^{\circ}}     %% short for \lambda circ

\def\lhs#1#2{{_\bullet\!\!}\left\langle #1,#2\right\rangle} % Algebra valued left inner product
\def\rhs#1#2{\left\langle #1,#2\right\rangle\!\!{_\bullet}}	% Algebra valued right inner product
\def\ft{S} % Frame or frame-type operator
\def\modan{\Phi}
\def\modsy{\Psi}
\def\modft{\Theta}

\def\tfp#1{#1 \times \widehat{#1}} % Time-frequency plane associated to the group #1
\def\La{\Lambda} % Lattice in G x G^

\usepackage[foot]{amsaddr}

\author{Are Austad}
\author{Mads S. Jakobsen}
\author{Franz Luef}
\address[Are Austad, Franz Luef]{Norwegian University of Science and Technology, Department of Mathematical Sciences, Trondheim, Norway.}
\email{are.austad@ntnu.no, mja@weibel.dk, franz.luef@ntnu.no}
\address[Mads S. Jakobsen]{Weibel Scientific A/S, Solvang 30, 3450 Allerød, Denmark}

\title{Gabor Duality Theory for Morita Equivalent $C^*$-algebras}
\begin{document}
	\maketitle
	
	\begin{abstract}
		The duality principle for Gabor frames is one of the pillars of Gabor analysis. We establish a far-reaching generalization to Morita equivalent $C^*$-algebras where the equivalence bimodule is a finitely generated projective Hilbert $C^*$-module. These Hilbert $C^*$-modules are equipped  with some extra structure and are called Gabor bimodules. We formulate a duality principle for standard module frames for Gabor bimodules which reduces to the well-known Gabor duality principle for twisted group $C^*$-algebras of a lattice in phase space. We lift all these results to the matrix algebra level and in the description of the module frames associated to a matrix Gabor bimodule we introduce $(n,d)$-matrix frames, which generalize superframes and multi-window frames. Density theorems for $(n,d)$-matrix frames are established, which extend the ones for multi-window and super Gabor frames. Our approach is based on the localization of a Hilbert $C^*$-module with respect to a trace.% and the Cohen-Hewitt factorization theorem. 
	\end{abstract}
	\tableofcontents 
	\section{Introduction} 
	Hilbert $C^*$-modules are well-studied objects in the theory of operator algebras and Rieffel made the crucial observation that they provide the correct framework for the extension of Morita equivalence of rings to $C^*$-algebras. In his seminal work \cite{ri74-2} he noted that the equivalence bimodules between two $C^*$-algebras are bimodules where the left and right Hilbert $C^*$-module structures are compatible and the respective $C^*$-valued inner products satisfy an associativity condition. We are interested in the features of these equivalence bimodules from the perspective of frame theory. In \cite{frla02} the notion of standard module frame was introduced for countably generated Hilbert $C^*$-modules. Rieffel has already in \cite{ri88} observed that finitely generated equivalence bimodules may be described in terms of finite standard module frames and used it in his study of Heisenberg modules, which is a class of projective Hilbert $C^*$-modules over twisted group $C^*$-algebras. In \cite{jalu18duality} the properties of standard module frames for Heisenberg modules have been studied from the perspective of duality theory, which was motivated by the observation in \cite{lu09} that these module frames are closely related to Gabor frames for an associated Hilbert space. 

Gabor frames have some additional features not shared by wavelets and shearlets that is due to the seminal contributions \cite{dalala95,ja95,rosh97}, where they developed the duality theory of Gabor frames. 
\begin{theorem*}[Duality Theorem]
  The Gabor system $\{ e^{2\pi i\beta lt}g(t-\alpha k)\}_{k,l\in\Z}$ generated by a function $g\in L^2(\mathbb{R})$ is a frame for $L^2(\mathbb{R})$ if and only if $\{ e^{2\pi ilt/\alpha }g(t-k/\beta)\}_{k,l\in\Z}$ is a Riesz sequence for the closed span of $\{ e^{2\pi ilt/\alpha }g(t-k/\beta)\}_{k,l\in\Z}$ in $L^2 (\R)$ Here $t \in \R$.
\end{theorem*}	
	
Due to its far-reaching implications there have been attempts to extend the duality principle to other classes of frames \cite{ab10,grly09,barbieri2015riesz,barbieri2018invariant}, see \cite{cakula05,chxizh13,chst15,chst16} for the theory of R-duals and \cite{fan2016}.
	
	Motivated by the link between the duality theory of Gabor frames and the Morita equivalence of noncommutative tori \cite{lu09,jalu18duality} we extend the duality theory of Gabor frames to module frames for equivalence bimodules between Morita equivalent $C^*$-algebras. %We use the localization of Hilbert $C^*$-modules with respect to a positive linear functional to link the module frame statements to ones for frames in the corresponding Hilbert spaces.
The setup for our duality theory has its roots in \cite{jalu18duality} and is as follows: Let $A$ and $B$ be $C^*$-algebras where $B$ is assumed to have a unit and be equipped with a faithful finite trace $\tr_B$.
	We define a \textit{left Gabor bimodule} to be a quadruple
	\begin{equation}
	%\label{Equation: Gabor bimodule}
	(A, B, E, \tr_B)
	\end{equation}
	where $E$ is a Morita equivalence $A$-$B$-bimodule.
	%\begin{enumerate}
	%\item $B$ is unital and $\tr_B$ is a faithful trace on $B$.
	  %\item There exist $\A \subset A$ and $\B \subset B$ dense Banach $*$-subalgebras such that the enveloping $C^*$-algebra of $\A$ is $A$, and the enveloping $C^*$-algebra of $\B$ is $B$.
	 %\item $E$ is a Morita equivalence $A$-$B$-bimodule.
	 %\item %We induce a (possibly unbounded) trace 
	 %$\tr_A$ is a (possibly unbounded) induced trace on $A$ as determined by $E$, see \cref{prop:bimodule_localization}.
	 %\item $\A \subset \Dom (\tr_A)$.
	 %\item There exists a $\A$-$\B$-pre-equivalence bimodule $\E \subset E$ such that $\A \E \subset \E$, $\E \B \subset \E$, $\blangle \E , \E \rangle \subset \A$, and $\langle \E , \E \brangle \subset \B$.
%\end{enumerate}		
We show that module frames for Gabor bimodules admit a duality theorem and by localization with respect to a trace we are able to connect these module frame statements to results on frames in Hilbert spaces. Note that in \cite{fogr05} a different notion of localization of frames was introduced which constructs frames with additional regularity, which we also establish in our general setting.
	
The main application of our duality results is a concise treatment of Gabor frames for closed cocompact subgroups of locally compact abelian phase spaces. Our general approach to duality principles has led us to the introduction of $(n,d)$-matrix Gabor frames that is a joint generalization of multi-window superframes and Riesz bases and we prove that their Gabor dual systems are $(d,n)$-matrix Gabor frames.
	
Let $G$ be a second countable LCA group and let $\La$ be a closed subgroup of $G \times \widehat{G}$. For $g \in L^2 (G \times \Z_n \times \Z_d)$, let
\begin{equation}
\GS (g; \La) := \{\pi(\la) g_{i,j} \mid \la \in \La   \}_{i\in \Z_n, j\in \Z_d}.
\end{equation}
We say $\GS(g;\La)$ is an \textit{$(n,d)$-matrix Gabor frame for} $L^2 (G)$ if the collection of time-frequency shifts $\GS(g;\La)$ is a frame for $L^2 (G \times \Z_n \times \Z_d)$. %We then say that $\GS (g; \La)$ is an $(n,d)$-matrix Gabor frame for $L^2 (G)$. % This will be our preferred term.
	Equivalently, there exists $h \in L^2 (G \times \Z_n \times \Z_d)$ such that for all $f \in L^2 (G \times \Z_n \times \Z_d)$ we have
	\begin{equation}
	%\label{Equation: (n,d)-matrix frames}
	f_{r,s} = \sum_{k\in \Z_d} \sum_{l \in \Z_n} \int_{\La} \langle f_{r,k} , \pi (\la)g_{l,k}\rangle_{L^2 (G)}\pi (\la)h_{l,s} d\la,
	\end{equation}
	for all $r \in \Z_n$ and $s \in \Z_d$.
We develop the theory of these $(n,d)$-matrix Gabor frames and prove a duality theorem for this novel type of frames.

Let us summarize the content of this paper. In \cref{Section: Preliminaries}	we collect some facts about $C^*$-algebras, Hilbert $C^*$-modules, %the Cohen-Hewitt factorization theorem, 
the localization of Hilbert $C^*$-modules and finitely generated projective Hilbert $C^*$-modules. In \cref{Section: Abstract Gabor Analysis} we introduce Gabor bimodules, %which are equivalence bimodules with some extra features, 
and study the case when these have a single generator in terms of module frames. We establish the analog of the duality theorem for  Gabor frames for Gabor bimodules with one generator. In \cref{Subsection: Extending to several generators} we extend all of these results to the case of finitely many generators which leads one naturally to matrix-valued extensions of the statements and definitions in the preceding section. We also prove a density theorem for module frames. In the final section, \cref{Subsection: Implications for Gabor Analysis}, we discuss applications to Gabor frames for closed subgroups of the time-frequency plane of locally compact abelian groups.

\section{Preliminaries on $C^*$-algebras and Hilbert $C^*$-modules}
	\label{Section: Preliminaries}
   We assume basic knowledge about Banach $*$-algebras, $C^*$-algebras, and of Banach modules and Hilbert $C^*$-modules. 
	In this section we collect definitions and basic facts of concepts crucial for this paper, such as positivity in $C^*$-algebras, Morita equivalence of $C^*$-algebras, and localization of Hilbert $C^*$-modules. For these topics we refer to \cite{lance1995hilbert}, \cite{rawi98}, and \cite{murphy2014}.
	
	For a $C^*$-algebra $A$ and $a \in A$, we denote by $\sigma_A (a)$ the spectrum of $a$ in $A$. We will need the following important result.
	\begin{prop}[\cite{murphy2014}, Theorem 2.1.11]
		\label{Proposition: C*-subalgebra inverse closed}
		Let $A$ be a unital $C^*$-algebra and let $B$ be a $C^*$-subalgebra of $A$ containing the unit of $A$. If $b \in B$, then $\sigma_B (b) = \sigma_A (b)$. Equivalently, if $b$ is invertible in $A$, then $b^{-1} \in B$. 
	\end{prop}
	\begin{defn}
		Let $A$ be a unital $C^*$-algebra and let $B \subset A$ be a Banach $*$-subalgebra of $A$ with the same unit. We say that $B$ is \textit{spectral invariant} in $A$ if whenever $b \in B$ is invertible with $b^{-1} \in A$, we have $b^{-1}\in B$. 
	\end{defn}
	Now recall that a selfadjoint element $a$ in a $C^*$-algebra with $\sigma_A (a) \subset [0,\infty)$ is called \textit{positive}. We state a useful characterization of positivity.
	\begin{prop}
		\label{Proposition: Positivity in C*-alg}
		Let $A$ be a $C^*$-algebra. For $a= a^* \in A$ we have
		$\sigma_A (a) \subset [0,\infty)$ if and only if $a=b^* b$ for some $b \in A$. 
	\end{prop}
	Denote by $A^{+}$ the set of positive elements in the $C^*$-algebra $A$. The positive elements form a cone. In particular, if $a \in A^{+}$ then $ka \in A^{+}$ for all $k \in [0,\infty)$, and if $a_1,a_2 \in A^{+}$ then $a_1 + a_2 \in A^{+}$. We also obtain a partial order on $A^{+}$ by $a \leq b$ if and only if $b-a \in A^{+}$. Note that not all elements of $A^{+}$ are comparable, but all elements are comparable to $1_A$ in the case $A$ is unital. 
	
	Central to our results in \cref{Section: Abstract Gabor Analysis} will be the localization of a Hilbert $C^*$-module. For this we need positive linear functionals.
	\begin{defn}
		A positive linear functional on a $C^*$-algebra $A$ is a linear functional $\phi$ such that $\phi (A^{+}) \subset [0,\infty)$. If $\Vert \phi \Vert = 1$ we say $\phi$ is a \textit{state}.% In case $A$ is unital, it is known that $\phi$ is a state if and only if $\phi (1_A) = 1$.
	\end{defn}
	\begin{rmk}
	    If $\phi \colon A \to \C$ is a positive linear functional and $A$ is unital, it is known that $\phi$ is a state if and only if $\phi (1_A) = 1$.
	\end{rmk}

	We will denote the set of adjointable operators on the Hilbert $A$-module $E$ by $\End_A (E)$, and the set of compact module  operators by $\K(E)$. The following two results will be of great importance in our approach to duality theorems. 
	\begin{prop}[\cite{lance1995hilbert}, Proposition 1.1]
		\label{Proposition: Middle positivity}
		Let $A$ be a $C^*$-algebra. If $E$ is an inner product $A$-module and $f,g \in E$, then
		\begin{equation*}
		{}_{A}\langle g,f \rangle {}_{A}\langle f,g \rangle \leq \Vert {}_{A}\langle f,f\rangle \Vert {}_{A}\langle g,g \rangle,
		\end{equation*}
		where ${}_{A}\langle\cdot,\cdot\rangle$ is the $A$-valued inner product.
		Also, whenever $c \geq 0$ in $A$, we have $a^* c a \leq \Vert c \Vert a^*a$ for all $a \in A$. 
	\end{prop}
	\begin{prop}[\cite{rawi98}, Corollary 2.22]
		\label{Proposition: Inner product positive operator}
		Let $A$ be a $C^*$-algebra. If $E$ is a Hilbert $A$-module and $T \in \End_A (E)$, then for any $f \in E$
		\begin{equation*}
		{}_{A}\langle Tf , Tf \rangle \leq \Vert T \Vert^2 {}_{A}\langle f,f \rangle
		\end{equation*}
		as elements of the $C^*$-algebra $A$, where ${}_{A}\langle\cdot,\cdot\rangle$ is the $A$-valued inner product.
	\end{prop}
	Suppose $\phi$ is a positive linear functional on a $C^*$-algebra $B$, and that $E$ is a right Hilbert $B$-module.  We define an inner product
	\begin{equation*}
	\begin{split}
	\langle \cdot , \cdot\rangle_{\phi}: E \times E &\to \C \\
	(f,g) &\mapsto \phi (\langle g,f \rangle_{B} ),
	\end{split}
	\end{equation*}
	where $\langle \cdot , \cdot\rangle_B$ is the $B$-valued inner product. We may have to factor out the subspace
	\begin{equation*}
	N_{\phi} := \{ f \in E \mid \langle f,f \rangle_{B} = 0\},
	\end{equation*}
	and complete $E/N_{\phi}$ with respect to $\langle\cdot , \cdot\rangle_{\phi}$. This yields a Hilbert space which we will denote by $H_E$. This is known as the localization of $E$ in $\phi$. There is a natural map $\rho_{\phi} : E\to H_E$ which induces a map $\rho_{\phi} :\End_B (E) \to \mathbb{B} (H_E)$. We will focus entirely on the case in which $\phi$ is a faithful positive linear functional, that is, when $b\in B^{+}$ and $\phi (b)=0$ implies $b=0$. In that case $N_{\phi} = \{0\}$ and we have the following useful result from \cite[p. 57-58]{lance1995hilbert}.
	\begin{prop}
		\label{Localization norm preserved}
		Let $A$ be a $C^*$-algebra equipped with a faithful positive linear functional $\phi: A \to \C$, and let $E$ be a left Hilbert $A$-module. Then the map $\rho_{\phi}: \End_A (E) \to \mathbb{B}(H_E)$ is an injective $*$-homomorphism.
	\end{prop}
	
	The Hilbert $C^*$-modules of interest will be of a very particular form in that they will be $A$-$B$-equivalence bimodules for $C^*$-algebras $A$ and $B$. We will denote the $A$-valued inner product by $\blangle\cdot , \cdot\rangle$, and the $B$-valued inner product by $\langle\cdot , \cdot\brangle$. 
	\begin{defn}
		\label{Definition: Equivalence bimodule}
		Let $A$ and $B$ be $C^*$-algebras. A \emph{Morita equivalence bimodule between} $A$ and $B$, or an \emph{$A$-$B$-equivalence bimodule}, is a Hilbert $C^*$-module $E$ satisfying the following conditions.
		\begin{enumerate}
			\item $\overline{\blangle E,E\rangle}=A$  and $\overline{\langle E , E \brangle} = B$, where $\blangle E,E\rangle = \linspan_\C \{ \blangle f,g \rangle \mid f,g\in E\}$ and likewise for $\langle E,E\brangle$.
			\item For all $f,g \in E$, $a \in A$ and $b \in B$,
			\begin{equation*}
			\text{$\langle a f, g\brangle = \langle f, a^* g \brangle$  and $\blangle fb, g\rangle = \blangle f,gb^*\rangle$.}
			\end{equation*}
			\item For all $f,g,h \in E$,
			\begin{equation*}
			\blangle f,g \rangle h = f\langle g,h \brangle.
			\end{equation*}
		\end{enumerate}
		
		Now let $\A \subset A$ and $\B \subset B$ be dense Banach $*$-subalgebras such that the enveloping $C^*$-algebra of $\A$ is $A$, and the enveloping $C^*$-algebra of $\B$ is $B$. Suppose further there is a dense $\A$-$\B$-inner product submodule $\E \subset E$ such that the conditions above hold with $\A,\B,\E$ instead of $A,B,E$. In that case we say $\E$ is an \textit{$\A$-$\B$-pre-equivalence bimodule}.
	\end{defn}
	We will make repeated use of the following fact in the sequel without mention.
	\begin{prop}[\cite{rawi98}, Proposition 3.11]
		Let $A$ and $B$ be $C^*$-algebras and let $E$ be an $A$-$B$-equivalence bimodule. Then $\Vert \blangle f,f \rangle \Vert = \Vert \langle f,f \brangle \Vert$ for all $f\in E$.
	\end{prop}

	It is a well-known result that if $E$ is an $A$-$B$-equivalence bimodule, then $B \cong \K_A (E)$ through the identification $\Theta_{f,g}\mapsto \langle f,g \brangle$. Here $\Theta_{f,g}$ is the compact module operator $\Theta_{f,g}: h \mapsto \blangle h, f\rangle g$. We make particular note of the case when $E$ is a finitely generated Hilbert $A$-module. 
	\begin{prop}
		\label{Proposition: Fin gen proj iff unital}
		Let $E$ be an $A$-$B$-equivalence bimodule. Then $E$ is a finitely generated projective $A$-module if and only if $B$ is unital.
	\end{prop}
	\begin{proof}
		Suppose first $A$ is finitely generated and projective as a Hilbert $A$-module. As $E$ is finitely generated, any $A$-endomorphism on $E$ is determined by its action on a finite set of generators. Hence $\End_A (E) = \K_A (E)$, and the former is unital. Since $B \cong \K_A (E)$, have that $B$ is unital.
		
		Conversely, we assume that $B$ is unital. As $B \cong \K_A (E)$, and the latter is an ideal in $\End_A (E)$, it follows that $\K_A (E) = \End_A (E)$. As $B$ is unital and $\langle E,E \brangle$ is dense in $B$, we can find elements $f_1 , \ldots , f_n , g_1 , \ldots , g_n \in E$ such that $\sum_{i=1}^n \langle f_i , g_i \brangle = 1_B$. The maps
		\begin{equation*}
		\begin{split}
		s: E &\to A^n \\
		h &\mapsto (\blangle h , f_i \rangle)_{i=1}^n,
		\end{split}
		\end{equation*}
		and 
		\begin{equation*}
		\begin{split}
		r: A^n &\to E\\
		(a_i)_{i=1}^n &\mapsto \sum_{i=1}^n a_i g_i,
		\end{split}
		\end{equation*}
		are $A$-module maps satisfying
		\begin{equation*}
		r\circ s (z) = \sum_{i=1}^n \blangle h,f_i\rangle g_i = \sum_{i=1}^n h \langle f_i , g_i \brangle = h \sum_{i=1}^n \langle f_i , g_i \brangle = h \cdot 1_B = h
		\end{equation*}
		for all $h \in E$. It follows that $E$ is a finitely generated projective $A$-module. 
	\end{proof}
	Note that the systems $\{f_1,...,f_n\}$ and $\{g_1,...,g_n\}$ are not necessarily $A$-linearly independent, but they still provide a reconstruction formula: $z= \sum_{i=1}^n \blangle h,f_i\rangle g_i$. Motivated by spanning sets in finite-dimensional vector spaces, also called frames, we call the system 
$\{f_1,...,f_n\}$ a {\it %standard 
	module frame} for $E$ and $\{g_1,...,g_n\}$ is referred to as a {\it dual module frame}. The properties of %standard 
module frames for equivalence bimodules are the main objective of this work.

%First we reformulate the frame condition for modules frames with one generator. \todo{Reformulate}
The following two results concern properties of module frames consisting of a single element, though we do not formally introduce module frames until \cref{Def:Single-module-frame}. For our setup it will turn out that it is enough to consider module frames consisting of only one element, see \cref{Section: Abstract Gabor Analysis}. The results will come into play when we relate module frames and Gabor frames in \cref{Subsection: Implications for Gabor Analysis}.
%We will be mostly interested in the case where the module frame consists of a single element. Indeed, we will show in \cref{Section: Abstract Gabor Analysis} that in our setup we can always reduce to this case. When the module frame consists of a single element
\begin{lemma}
	\label{Lemma: Modular frame bounds reciprocal}
	Let $A$ be any $C^*$-algebra, and let $E$ be a (left) Hilbert $A$-module. Suppose $T \in \End_A (E)$ is such that there exist $C,D > 0$ such that
	\begin{equation}
	\label{Equation: Positive invertible operator inner product inequality}
	C \blangle f,f \rangle \leq \blangle Tf,f \rangle \leq D \blangle f,f\rangle,
	\end{equation}
	for all $f \in E$. Then $T$ is invertible, and 
	\begin{equation*}
	\frac{1}{D} \blangle f,f \rangle \leq \blangle T^{-1} f,f \rangle \leq \frac{1}{C} \blangle f,f\rangle,
	\end{equation*}
	for all $f \in E$.
\end{lemma}
\begin{proof}
	By \eqref{Equation: Positive invertible operator inner product inequality}, we see that $T$ is positive and invertible with $C \Id_E \leq T \leq D \Id_E$. Positivity is preserved when multiplying by positive commuting operators, so it follows that $C T^{-1} \leq T^{-1} T \leq DT^{-1}$, from which we get $\frac{1}{D} \Id_E \leq T^{-1} \leq \frac{1}{C} \Id_E$.
\end{proof}
\begin{lemma}
	\label{Lemma: Optimal frame bounds for modular frame}
	Let $A$ be any $C^*$-algebra, and let $E$ be a (left) Hilbert $A$-module. Let $T \in \End_A (E)$ be such that there exist $C,D > 0$ such that
	\begin{equation}
	C \blangle f,f \rangle \leq \blangle Tf,f \rangle \leq D \blangle f,f\rangle.
	\end{equation}
	for all $f \in E$.
	Then the smallest possible value of $D$ is $\Vert T \Vert$, and the largest possible value for $C$ is $\Vert T^{-1} \Vert^{-1}$.
\end{lemma}
\begin{proof}
	Since $T$ is positive we have $\Vert T \Vert = \sup_{\Vert f \Vert = 1} \{ \Vert\blangle Tf,f\rangle\Vert \}$. It follows that the smallest value for $D$ is $\Vert T \Vert$. By \cref{Lemma: Modular frame bounds reciprocal} we see by the same argument that the minimal value for $\frac{1}{C}$ is $\Vert T^{-1} \Vert$. Hence the largest value for $C$ is $\Vert T^{-1} \Vert^{-1}$.
\end{proof}

%Note also that it is known that if $E$ is an $A$-$B$-equivalence bimodule and $f \in E$, then $\Vert \blangle f,f \rangle \Vert = \Vert \langle f,f \brangle \Vert$, see for example Proposition 3.11 of \cite{rawi98}. We make use of this below without mention. 

It is an interesting question when the property of being finitely generated projective passes to dense subalgebras and corresponding dense submodules. %We will make repeated use of the following. 
\begin{prop}
	\label{Proposition: Fin gen proj passes to subalg when spectral invariant}
	Let $E$ be an $A$-$B$-equivalence bimodule as in \cref{Definition: Equivalence bimodule}, with $B$ unital. Suppose there are dense Banach $*$-subalgebras $\A \subset A$ and $\B \subset B$, where $\B$ is spectral invariant in $B$ and has the same unit as $B$. Suppose further that $\E \subset E$ is an $\A$-$\B$-pre-equivalence bimodule. If $E$ is a finitely generated projective $A$-module, then $\E$ is a finitely generated projective $\A$-module.
	%Suppose further that there is a dense $\A$-$\B$-submodule $\E$ of $E$ satisfying $\A \E \subset \E$, $\E \B \subset \E$, $\blangle \E , \E \rangle \subset \A$, and $\langle \E , \E \brangle \subset \B$. If $E$ is a finitely generated projective $A$-module, then $\E$ is a finitely generated projective $\A$-module.
\end{prop}
\begin{proof}
	The latter part of the proof of \cref{Proposition: Fin gen proj iff unital} can be adapted to this situation. The full proof can be found in \cite[Proposition 3.7]{ri88}.
\end{proof}
Since we aim to mimic the situation of Gabor analysis, which we will treat in \cref{Subsection: Implications for Gabor Analysis}, the positive linear functional which we localize our Morita equivalence bimodule with respect to will have a particular form. In particular it will be a faithful trace. For unital Morita equivalent $C^*$-algebras $A$ and $B$ Rieffel showed in \cite{ri81c} that there is a bijection between non-normalized finite traces on $A$ and non-normalized finite traces on $B$ under which to a trace $\tr_B$ on $B$ there is an associated trace $\tr_A$ on $A$ satisfying
\begin{equation}
\label{Equation: Trace relation}
\tr_A (\blangle f,g \rangle ) = \tr_B (\langle g,f \brangle)
\end{equation}
for all $f,g \in E$. Here $E$ is the Morita equivalence bimodule. We will in the sequel almost always consider $A$ or $B$ unital, and so instead we will suppose the existence of a finite faithful trace on one $C^*$-algebra (the unital one) and induce a possibly unbounded trace on the other $C^*$-algebra. % by ways of $\eqref{Equation: Trace relation}$. That is, if $B$ is a unital $C^*$-algebra with faithful trace $\tr_B$, we define
%\begin{equation*}
%	\tr_A (\blangle f,g \rangle) := \tr_B (\langle g,f \brangle)
%\end{equation*}
%for all $f,g \in E$.
The following was proved in \cite[Proposition 2.7]{AuEn19} and ensures that this procedure works.
\begin{prop}\label{prop:bimodule_localization}
	%Let $A$ and $B$ Morita equivalent $C^*$-algebras, 
	Let $E$ be an $A$-$B$-equivalence bimodule, and suppose $\tr_B$ is a faithful finite trace on $B$. %Furthermore, let $A^{\tr_A} = \linspan \{a \in A^+ \mid \tr_A (a) < \infty \}$, where $\tr_A$ is defined by ways of \eqref{Equation: Trace relation}. 
	Then the following hold:
	\begin{enumerate}
		\item [i)] %If $E$ is an imprimitivity $A$-$B$-bimodule, then 
		%$\tr = \tr_A$ is the unique lower semi-continuous trace on $A$ such that \todo{Check that this item makes sense logically in the way it is phrased.}
		There is a unique lower semi-continuous trace on $A$, denoted $\tr_A$, for which
		\begin{equation}
		\tr_A(\lhs{f}{g}) = \tr_B(\rhs{g}{f}) \label{eq:trace_compatible_unbounded}
		\end{equation}
		for all $f,g \in E$. Moreover, $\tr_A$ is faithful, and densely defined since it is finite on $\linspan \{ \lhs{f}{g} : f,g \in E \}$. Setting
		\begin{equation}
		\langle f,g \rangle_{\tr_A} = \tr_A(\lhs{f}{g}), \quad \langle f,g \rangle_{\tr_B} = \tr_B (\langle g,f\brangle), \label{eq:inner_product_unbounded}
		\end{equation}
		for $f,g \in E$ defines inner products on $E$, with $\langle f,g \rangle_{\tr_A} = \langle f,g \rangle_{\tr_B}$ for all $f,g \in E$. Consequently, the Hilbert space obtained by completing $E$ in the norm $\| f \|' = \tr_A(\lhs{f}{f})^{1/2}$ is just the localization of $E$ with respect to $\tr_B$.
		\item [ii)] If $E$ and $F$ are equivalence $A$-$B$-bimodules, then every adjointable $A$-linear operator $E \to F$ has a unique extension to a bounded linear operator $H_{E} \to H_{F}$. Furthermore, the map $\End_A(E,F) \to \End(H_{E},H_{F})$ given by sending $T$ to its unique extension is a norm-decreasing linear map of Banach spaces. Finally, if $E = F$, the map $\End_A(E) \to \mathbb{B}(H_{E})$ is an isometric $*$-homomorphism of $C^*$-algebras.
	\end{enumerate}
\end{prop}
 If both $C^*$-algebras are unital then the induced trace is also a finite trace as in \cite{ri81c}, see \cite[p. 8]{AuEn19}. 
 
 %The following convention will be in place throughout the paper unless otherwise mentioned.
 \begin{conv}
 We have the following as a standing assumption for the rest of the manuscript unless otherwise specified: Suppose we have a faithful trace $\tr_B$ on a unital $C^*$-algebra $B$, and an $A$-$B$-equivalence bimodule $E$. If $A$ is a unital $C^*$-algebra, we pick the normalization of $\tr_B$ such that $\tr_A$ becomes a state. That is, we normalize the trace on the $C^*$-algebra on the left in the Morita equivalence.
 \end{conv}

%The Cohen--Hewitt factorization theorem is going to be a useful tool in our arguments.
%\begin{prop}[Cohen-Hewitt Factorization]
%	\label{Prop: Cohen Factorization}
%	Suppose $A$ is a Banach algebra with left approximate unit $(e_j)_j$ and let $E$ be a left Banach $A$-module. If $\lim_j e_j f = f$ for $f \in E$ then we can write $f= ag$ for some $a \in A$ and some $g\in E$. The analogous statement holds for right modules.
%\end{prop}
%\begin{rmk}
%	Recall that a non-unital $C^*$-algebra $A$ admits a two-sided and uniformly norm bounded approximate identity $(e_j)_{j\in J}$ such that $\lim_j e_j f = f$ for all $f$ in a Hilbert $C^*$-module $E$ over $A$. Since Hilbert $C^*$-modules are also Banach modules, we have the Cohen-Hewitt factorization  Hilbert $C^*$-modules at our disposal. 
%\end{rmk}

\section{ Gabor bimodules}
\label{Section: Abstract Gabor Analysis}
\subsection{The single generator case}
\label{Subsection: The Single Generator Case}
Throughout this section we discuss properties of equivalence bimodules of the following type.
\begin{defn}
	\label{Definition: Gabor bimodule}
Let $A$ and $B$ be $C^*$-algebras where $B$ is assumed to have a unit and is equipped with a faithful finite trace $\tr_B$.
	We define a \textit{left Gabor bimodule} to be a quadruple
	\begin{equation}
	\label{Equation: Gabor bimodule}
	(A, B, E, \tr_B)
	\end{equation}
    where $E$ is an $A$-$B$-equivalence bimodule.
	  %\item There exist dense Banach $*$-subalgebras $\A \subset A$ and $\B \subset B$ such that the enveloping $C^*$-algebra of $\A$ is $A$, and the enveloping $C^*$-algebra of $\B$ is $B$.
	 %\item There exists a $\A$-$\B$-pre-equivalence bimodule $\E \subset E$ such that $\A \E \subset \E$, $\E \B \subset \E$, $\blangle \E , \E \rangle \subset \A$, and $\langle \E , \E \brangle \subset \B$.
	 
We define a \textit{right Gabor bimodule} analogously, that is, it is a quadruple
    \begin{equation}
	(A, B, E, \tr_A)
	\end{equation}
	where $A$ is unital, $\tr_A$ is a faithful finite trace on $A$, and $E$ is an $A$-$B$-equivalence bimodule.
\end{defn}
\begin{rmk}
    By \cref{prop:bimodule_localization} we may always induce a possibly unbounded trace $\tr_A$ on $A$ given a left Gabor bimodule $(A, B, E, \tr_B)$. Indeed, this will be of great importance in the sequel, and we will use $\tr_A$ without mentioning that it is induced by $\tr_B$. Likewise with $\tr_B$ if we treat the case of right Gabor bimodules.
\end{rmk}
We are also interested in Gabor bimodules possessing some regularity, see \cref{Subsection: Implications for Gabor Analysis}.
\begin{defn}
	A \textit{left Gabor bimodule with regularity} is a %Gabor bimodule $(A, B, E, \tr_A, \tr_B)$ along with $\A,\B,\E$ as described above, that is, 
	septuple
	\begin{equation}
	\label{Equation: Gabor bimodule with regularity}
	(A, B, E, \tr_B , \A, \B, \E),
	\end{equation}
	such that
	\begin{enumerate}
		\item $(A, B, E, \tr_B)$ is a left Gabor bimodule.
		\item $\A \subset A$ and $\B \subset B$ are dense Banach $*$-subalgebras.
		\item $\E \subset E$ is an $\A$-$\B$-pre-equivalence bimodule.
		\item $\B$ is spectral invariant in $B$ with the same unit.
		%\item $\A \subset \Dom (\tr_A)$.
	\end{enumerate}
	We define a \textit{right Gabor bimodule with regularity} analogously. %when $A$ is unital (and $\A \subset A$ is spectral invariant with the same unit).
\end{defn}
The rest of this section will be devoted to exploring properties of Gabor bimodules, mostly left Gabor bimodules. 
In \cref{Subsection: Implications for Gabor Analysis} we show that Gabor bimodules over twisted group $C^*$-algebras for LCA groups are Rieffel's Heisenberg modules and provide a different approach to Gabor analysis. We start with some basic definitions from \cite{frla02}. We restrict to a single generator in this section, and extend the results to finitely many generators in \cref{Subsection: Extending to several generators}. Indeed, we will see in that section that even the case of finitely many generators can be reduced to the case of a single generator of an associated Morita equivalence bimodule. 
\begin{defn}
	\label{Definition: Modular synthesis and analysis}
	For $g\in E$ we define the \textit{analysis operator} by
	\begin{equation}
	\begin{split}
	\modan_{g}: E &\to A \\
	f &\mapsto \blangle f,g\rangle,
	\end{split}
	\end{equation}
	and the \textit{synthesis operator}:
	\begin{equation}
	\begin{split}
	\modsy_{g} : A &\to E\\
	a &\mapsto a\cdot g.
	\end{split}
	\end{equation}
\end{defn}
	An elementary computation shows that $\modan_{g}^* = \modsy_{g}$.

\begin{rmk}
	As $E$ is an $A$-$B$-bimodule, we could just as well have defined the analysis operator and the synthesis operator with respect to the $B$-valued inner product. Indeed we will need this later, but it will then be indicated by writing $\modan_g^{B}$. Unless otherwise indicated the analysis operator and synthesis operator will be with respect to the left inner product module structure.
\end{rmk}
\begin{defn}
	\label{Definition: Modular frame operator}
	For $g,h \in E$ we define the \textit{frame-like operator} $\modft_{g,h}$ to be
	\begin{equation}
	\begin{split}
	\modft_{g, h} : E &\to E \\
	f &\mapsto  \blangle f,g \rangle h.
	\end{split}
	\end{equation}
	In other words, $\modft_{g,h} = \modsy_{h} \modan_{g} = \modan_{h}^*\modan_{g}$. The \textit{frame operator of} $g$ is the operator
	\begin{equation}
	\begin{split}
	\modft_{g} := \modft_{g, g} = \modan_{g}^* \modan_{g}: E &\to E \\
	f &\mapsto  \blangle f,g \rangle g.
	\end{split}
	\end{equation}
\end{defn}
\begin{rmk}
	The module frame operator $\modft_{g}$ is a positive operator since $\modft_{g} = \modan_{g}^* \modan_{g}$.
\end{rmk}
\begin{defn}
\label{Def:Single-module-frame}
 We say $g\in E$ generates a \textit{(single) module frame for $E$} if $\modft_{g}$ is an invertible operator $E\to E$. Equivalently, there exist constants $C,D >0$ such that
	\begin{equation}
	\label{Equation: Modular frame inequality}
	C\blangle f,f \rangle \leq  \blangle f,g \rangle \blangle g ,f\rangle \leq D\blangle f,f \rangle,
	\end{equation}
	holds for all $f \in E$. %Furthermore, we say $g$ is a \textit{modular Bessel vector for $E$} if only the upper inequality of Equation \eqref{Equation: Modular frame inequality} holds. The minimal $D>0$ such that the inequality holds true is called the \text{Bessel bound for $g$}.
\end{defn}
\begin{rmk}
	When $\{g\}$ is a module frame for $E$, $\modft_{g}$ is a positive invertible operator on $E$.
\end{rmk}
What follows will largely be a study of $g$ and $h$ in $E$ such that $\modft_{g, h}$ is invertible, and how this relates to inner product inequalities for a localization $H_E$ of $E$. Our interest in this question is due to the fact that in certain cases module frames can be localized to obtain Hilbert space frames, see \cref{Subsection: Implications for Gabor Analysis}. We begin with a result which generalizes the Wexler-Raz biorthogonality condition for Gabor frames, which we also look at in \cref{Subsection: Implications for Gabor Analysis}.
\begin{prop}[Wexler-Raz for Gabor modules]
	\label{Proposition: Gabor module Wexler-Raz}
	Let $g,h$ in $E$. Then $f= \modft_{g,h} f = \modft_{h,g} f$ for all $f \in E$ if and only if $B$ is unital and $ \langle g,h\brangle = \langle h,g\brangle= 1_B$.
\end{prop}
\begin{proof} 
	Suppose $f= \modft_{h,g} f = \modft_{g,h} f$ for all $f \in E$. Then $E$ is a finitely generated projective $A$-module as it is generated by $g$, and we can use $g$ and $h$ to make the maps $r$ and $s$ from the proof of \cref{Proposition: Fin gen proj iff unital}. Hence $B \cong \K_A (E) = \End_A (E)$, and as $\End_A (E)$ is unital, we deduce $B$ is unital. By Morita equivalence
	\begin{equation*}
	f= \modft_{g,h} f = \blangle f,g \rangle h = f\langle g , h \brangle = f  \langle g , h \brangle 
	\end{equation*}
	for all $f \in E$. Since $B$ acts faithfully on $E$ we deduce $\langle g , h\brangle  = 1_B$. Then also $$
	\langle h , g \brangle  =  \langle g , h \brangle^* = 1_B^* = 1_B.
	$$ 
	
	Conversely, suppose $B$ is unital and $\langle g , h \brangle =  \langle h , g \brangle = 1_B$. Then
	\begin{equation*}
	f = f 1_B = f  \langle g , h \brangle = \blangle f,g \rangle h = \modft_{g, h} f,
	\end{equation*}
	and
	\begin{equation*}
	f = f1_B = f  \langle h , g \brangle = \blangle f, h \rangle g = \modft_{h, g} f,
	\end{equation*}
	which finishes the proof.
\end{proof}
The following result showcases a duality between certain $A$-submodules of $E$ and $B$-submodules of $E$, which is very particular to our setting of Morita equivalence bimodules.
\begin{prop}
	\label{Proposition: Dual atom reconstruction subspaces} 
	For any $g,h \in E$ the following two statements are equivalent. 
	\begin{enumerate}
		\item[(i)] $\blangle f, h \rangle g = f$ for all $f \in \overline{Ag}$.
		\item[(ii)] $f = g \langle h,f\brangle$ for all $f \in \overline{gB}$.
	\end{enumerate}
\end{prop}
\begin{proof}
	Suppose first $\blangle f,h\rangle g = f$ for all $f\in \overline{Ag}$. By Morita equivalence of $A$ and $B$,
	\begin{equation*}
	f = f \langle h,g\brangle
	\end{equation*}
	for all $f \in \overline{Ag}$, hence $\langle h,g\brangle$ fixes all elements in $\overline{Ag}$. In particular, since $A$ has an approximate unit, $g\in \overline{Ag}$, so $g\langle h,g\brangle = g$. Now let $f' \in gB$. We then write $f' = gb$ for some $b \in B$, and so we deduce 
	\begin{equation*}
	g\langle h,f'\brangle = g\langle h,gb\brangle = g\langle h,g \brangle b = gb = f',
	\end{equation*}
	since $g\langle h,g \brangle = g$ by the above. We extend the reconstruction formula to all of $\overline{gB}$ by continuity. The proof of the converse is completely analogous.%, but let us write it out. Suppose $f = g \langle h,f\brangle$ for all $f \in \overline{gB}$. By Morita equivalence,
	%\begin{equation*}
	%    f= \blangle g,h \rangle f,
	%\end{equation*}
	%for all $f \in \overline{gB}$. Since $B$ has an approximate unit, $g \in \overline{gB}$, and so in particular we have $\blangle g,h\rangle g = g$. We wish to prove that $f' = \blangle f' , h \rangle g$ for all $f' \in \overline{Ag}$. By Cohen factorization we write $f' = ag$ for some $a \in A$. Then
	%\begin{equation*}
	%    \blangle f' , h\rangle g = \blangle ag,h\rangle g = a \blangle g,h\rangle g = a g = g
	%\end{equation*}
	%by the above. The result follows.
\end{proof}
In the special case where \cref{Proposition: Dual atom reconstruction subspaces} (i) holds for all $f\in E$ we get another reconstruction formula. Note the (subtle) difference in the placement of $g$ and $h$ in the statement compared to statement (ii)  in the preceding proposition.
\begin{prop}
	\label{Proposition: $E$ to $hB$ duality}
	Let $g,h \in E$ be so that $\blangle f, h \rangle g = f$ for all $f \in E$. Then
	$$f = h \langle g,f\brangle \quad\text{for all $f \in \overline{hB}.$}$$
\end{prop}
\begin{proof}
	Suppose that $\blangle f,h\rangle g = f$ for all $f \in E$. Then $E$ is finitely generated and projective as an $A$-module as before, since it is singly generated by $g$, and we may use $g$ and $h$ to make the maps $r$ and $s$ from the proof of \cref{Proposition: Fin gen proj iff unital}. Hence $B \cong \K_A (E) = \End_A (E)$ and $B$ is unital. We may rewrite the equality to $f = f\langle h,g \brangle$ for all $f \in E$, which implies $\langle h,g \brangle = 1_B$ as $B$ acts faithfully on $E$. But then 
	$$
	\langle g,h\brangle = \langle h,g\brangle^* = 1_B^* = 1_B
	$$
	as well. Then if we let $f \in hB$ we may write $f=hb$ for some $b\in B$, and we get
	\begin{equation*}
	\begin{split}
	h\langle g,f \brangle = h\langle g,hb \brangle = h \langle g,h \brangle b = h 1_B b = hb = f.
	\end{split}
	\end{equation*}
	We extend the reconstruction formula to $\overline{hB}$ by continuity.
\end{proof}
Note that in the setting of \cref{Proposition: Dual atom reconstruction subspaces} we may of course interchange $g$ and $h$. However the subspaces $\overline{Ah}$ and $\overline{Ag}$ do not need to coincide. We may however guarantee $\overline{Ag} = \overline{Ah}$ when $h$ has a special form. 
\begin{lemma}
	\label{Lemma: Atom and canonical dual same span}
	Let $g\in E$ be such that $\modft_g\vert_{\overline{Ag}}$ is invertible as a map $\overline{Ag}\to \overline{Ag}$. For $h= \modft_g\vert_{\overline{Ag}}^{-1} g $ we have $\overline{Ag} = \overline{Ah}$.
\end{lemma}
\begin{proof}
	Let $f \in \overline{Ag}$. As $\modft_g$, and hence also $\modft_g \vert_{\overline{Ag}}$, is an $A$-module operator, so is $\modft_g \vert_{\overline{Ag}}^{-1}$. Thus we get
	\begin{equation*}
	f =  \modft_g \vert_{\overline{Ag}}^{-1} \modft_g  f = \modft_g \vert_{\overline{Ag}}^{-1} (\blangle f,g\rangle g) = \blangle f,g\rangle \modft_g \vert_{\overline{Ag}}^{-1} g = \blangle  f,g \rangle h \in \overline{Ah}.
	\end{equation*}
	Hence we have $\overline{Ag} \subset \overline{Ah}$. Also $g \in \overline{Ag}$ as $A$ has a left approximate unit, and as $\modft_g $ is invertible as a map $\overline{Ag} \to \overline{Ag}$ it follows that $h = \modft_g\vert_{\overline{Ag}}^{-1}g \in \overline{Ag}$. Hence $\overline{Ag} = \overline{Ah}$.
\end{proof}
In the remaining part of the article we focus mostly on the case where $\modft_{g,h}$ is invertible as a map $E\to E$.
\begin{defn}
	If $g \in E$ is such that $\modft_g$ is invertible, then $h = \modft_g^{-1}g$ is called the \textit{canonical dual atom of $g$}. 
\end{defn}
\begin{rmk}
	Note that if $g$ is such that $\modft_g: E \to E$ is invertible, then $\overline{Ag}=E$. To see this, let $f\in E$. Then 
	\begin{equation*}
	f = \modft_g \modft_g^{-1}f = \blangle \modft_g^{-1}f , g \rangle g \in \overline{Ag}.
	\end{equation*}
\end{rmk}
Given $g\in E$ such that $\modft_g:E \to E$ is invertible and with $h = \modft_{g}^{-1} g$ the canonical dual atom, one may ask what the canonical dual atom of $h$ is. The following lemma tells us that it is exactly what one would expect from Hilbert space frame theory.
\begin{lemma} 
	\label{Lemma: Frame operator of dual}
	Let $g\in E$ be such that $\modft_{g}:E\to E$ is invertible, and let $h = \modft_{g}^{-1}g$. Then $\modft_h g = h$. Moreover, $\modft_h:E  \to E$ is invertible with $\modft_h^{-1} = \modft_g $, and the canonical dual atom of $h$ is $g$.
\end{lemma}
\begin{proof}
	With $h=\modft_{g}^{-1}g$ the identity $\modft_h g = h$ is established as follows:
	\begin{equation*}
	\begin{split}
	\modft_h g &= \blangle g,h\rangle h = \blangle g, \modft_{g}^{-1} g \rangle \modft_{g}^{-1} g \\
	&= \modft_{g}^{-1} (\blangle \modft_{g}^{-1} g, g \rangle g ) = \modft_{g}^{-1} \modft_g \modft_{g}^{-1} g \\
	&= \modft_{g}^{-1} g = h.
	\end{split}
	\end{equation*}
	We proceed to show that $\modft_{h}$ is invertible. By \cref{Lemma: Atom and canonical dual same span} $\overline{Ag}=\overline{Ah}$, and we know $\overline{Ag}=E$. %By the Cohen-Hewitt factorization theorem any $f \in E$ may be expressed as $f = ag$ for some $a \in A$. 
	Now let $f \in Ag$, and write $f=ag$ for some $a \in A$. We then have
	%We may then do the following calculation.
	\begin{equation*}
	\begin{split}
	\modft_g \modft_h f &= \modft_g (\blangle f,h\rangle h) = \modft_g (\blangle ag, h \rangle h ) \\
	&= a \modft_g (\blangle g, \modft_{g}^{-1} g \rangle \modft_{g}^{-1} g )
	= a \modft_g \modft_{g}^{-1} (\blangle \modft_{g}^{-1} g, g \rangle g ) \\
	&=a \blangle \modft_{g}^{-1} g, g \rangle g
	= a \modft_g \modft_{g}^{-1} g \\
	&= ag = f.
	\end{split}
	\end{equation*}
	Likewise we have
	\begin{equation*}
	\begin{split}
	\modft_h \modft_g f &= \modft_h \modft_g (ag) = a \modft_h \modft_g (g) \\
	&= a \modft_h (\blangle g,g\rangle g) = a \blangle \blangle g,g \rangle g, h \rangle h \\
	&= a \modft_{g}^{-1} ( \blangle  \modft_{g}^{-1} (\blangle g,g \rangle g), g \rangle g) = a \modft_{g}^{-1} (\blangle \modft_{g}^{-1} \modft_g g,g \rangle g) \\
	&= a \modft_{g}^{-1} (\blangle g,g \rangle g) = a \modft_{g}^{-1} \modft_g g \\
	&= ag = f.
	\end{split}
	\end{equation*}
	$Ag$ is dense in $E$, and by extending the reconstruction formulas to all of $E$ by continuity it follows
	that $\modft_h^{-1} = \modft_g $. Then the canonical dual atom of $h$ is 
	\begin{equation*}
	\modft_h^{-1} h = \modft_g h = \blangle h,g\rangle g = \modft_g \modft_g^{-1}g = g,
	\end{equation*} which proves the result. 
\end{proof}
The following proposition tells us that $g$ and $\modft_g^{-1} g$ then indeed have the desired properties as described in \cref{Proposition: Dual atom reconstruction subspaces}.
\begin{prop}
	\label{Proposition: Canonical dual atom is dual atom}
	Let $g \in E$ be such that $\modft_g$ is invertible and let $h = \modft_g^{-1} g$. Then
	\begin{enumerate}
		\item $\blangle f, g \rangle h = f$ for all $f \in E = \overline{Ah} = \overline{Ag}$.
		\item $f'= h \langle g,f'\brangle$ for all $f' \in \overline{hB}$
		\item $\blangle f,h \rangle g = f$ for all $f \in E = \overline{Ah} = \overline{Ag}$.
	\end{enumerate}
\end{prop}
\begin{proof} 
	Note first that $E = \overline{Ag}= \overline{Ah}$ by \cref{Lemma: Atom and canonical dual same span}. Now,
	let $f \in Ah$ and write $f=ah$ for some $a \in A$. Then we have
	\begin{equation*}
	\begin{split}
	\blangle f,g\rangle h &= \blangle ah,g\rangle h = a \blangle h,g\rangle h = a \blangle \modft_g^{-1} g,  g \rangle h \\
	&= a \blangle g,\modft_g^{-1} g \rangle h  = a \blangle g,h\rangle h = a \modft_h g = ah = f,
	\end{split}
	\end{equation*}
	where we have used $\modft_h g= h$, which holds by \cref{Lemma: Frame operator of dual}. We extend the reconstruction formula by continuity so it is valid for all $f \in E$. By \cref{Proposition: Dual atom reconstruction subspaces} this also implies $f'= h \langle g,f'\brangle$ for all $f' \in \overline{hB}$. Hence items (1) and (2) are true. To show (3), let $f \in E$. Then
	\begin{equation*}
	\begin{split}
	\blangle f,h \rangle g &= \blangle f , \modft_g^{-1} g \rangle g = \blangle \modft_g^{-1} f , g\rangle g 
	= \modft_g \modft_g^{-1} f = f,
	\end{split}
	\end{equation*}
	so item (3) is also true. 
\end{proof}
We may also prove the following additional reconstruction formula when $h$ is the canonical dual atom. Note the (subtle) difference in where $h$ and $g$ are in the reconstruction formula compared to \cref{Proposition: Dual atom reconstruction subspaces}.
\begin{prop}
	Let $g \in E$ be such that $\modft_g$ is invertible, and let $h = \modft_g^{-1}g$. Then $f = h \langle g,f\brangle$ for all $f \in \overline{gB}$ and $f' = g\langle h,f'\brangle$ for all $f' \in \overline{hB}$. As a consequence, $\overline{gB} = \overline{hB}$.
\end{prop}
\begin{proof}
Suppose $f$ in $\overline{gB}$. For $g$ and $h$ as stated, we have $f = \blangle f,h\rangle g$ for all $f\in E$. By \cref{Proposition: Dual atom reconstruction subspaces} we then have $f = g\langle h,f\brangle$ for all $f \in \overline{gB}$. Then
	\begin{equation*}
	\begin{split}
	f&= g\langle h,f\brangle = \blangle g,h \rangle f = \blangle g, \modft_g^{-1}g \rangle f \\
	&= \blangle \modft_g^{-1}g ,g \rangle f = h \langle g,f \brangle.
	\end{split}
	\end{equation*}
	The second statement follows from noting that our assumptions imply $\overline{Ag} = \overline{Ah}=E$ by \cref{Lemma: Atom and canonical dual same span} and the fact that the canonical dual of $h$ is $g$ by \cref{Lemma: Frame operator of dual}. Then we may simply interchange $g$ and $h$ in the argument for the first assertion. 
	
	Lastly we prove $\overline{gB} = \overline{hB}$. We know $g \in \overline{gB}$ as $B$ has an approximate unit, so
	\begin{equation*}
	g = \modft_g h = h\langle g,g\brangle \in \overline{hB}.
	\end{equation*}
	Likewise, $h\in \overline{hB}$ and so
	\begin{equation*}
	h = \modft_h g = g\langle h,h\brangle \in \overline{gB}.
	\end{equation*}
	This finishes the proof.
\end{proof}
There is a correspondence between projections in Morita equivalent $C^*$-algebras, see for example \cite{ri88}. We formulate the following variant. Let $E$ be an $A$-$B$-equivalence bimodule, and let $B$ be unital. Then there is a way of constructing idempotents in $A$. This is the content of the following proposition.
\begin{prop}
	\label{Proposition: Identity induce idempotent}
	Let $E$ be an $A$-$B$-equivalence bimodule between a $C^*$-algebra $A$ and a unital $C^*$-algebra $B$. If $g,h \in E$ are such that $\langle g,h \brangle = 1_B$, then $\blangle g,h \rangle$ is an idempotent in $A$. In particular, the canonical dual atom $h = \modft_g^{-1} g$ yields a projection $\blangle g,h \rangle$ in $A$.
\end{prop}
\begin{proof}
	From $\langle g,h \brangle = 1_B = 1_B^* = \langle h,g \brangle$, we get
	\begin{equation*}
	\blangle g,h \rangle \blangle g,h \rangle = \blangle \blangle g,h \rangle g, h\rangle  = \blangle g \langle h,g \brangle , h \rangle = \blangle g \cdot 1_B , h \rangle = \blangle g,h\rangle,
	\end{equation*}
	so $\blangle g,h \rangle$ is an idempotent in $A$. If $h = \modft_g^{-1} g$, we also have
	\begin{equation*}
	\blangle g, h \rangle = \blangle g, \modft_g^{-1} g \rangle = \blangle \modft_g^{-1} g, g \rangle = \blangle h,g \rangle = \blangle g,h \rangle^*,
	\end{equation*}
	so $\blangle g,h \rangle$ is a projection in $A$.
\end{proof}
One of the cornerstones of Gabor analysis is the duality principle, see for example \cite{dalala95,ja95,rosh97}. One of the main intentions of this investigation is a reformulation of this duality principle in our module framework. To this end we introduce the following operator. For an element $g \in E$ we define the $B$-coefficient operator by
\begin{equation}
\begin{split}
\modan_g^B : E &\to B \\
f &\mapsto \langle g,f\brangle.
\end{split}
\end{equation}
Note that this operator is $B$-adjointable with adjoint
\begin{equation}
(\modan_g^{B})^* b \mapsto g\cdot b.
\end{equation}
We are now in the position to state and prove the module version of the duality principle.
\begin{prop}[Module Duality Principle]
	\label{Proposition: Modular Ron-Shen Duality}
	Let $g \in E$. The following are equivalent.
	\begin{enumerate}
		\item $\modft_g : E \to E$ is invertible.
		\item $\modan_g^B (\modan_g^{B})^*: B\to B$ is an isomorphism.
	\end{enumerate}
\end{prop}
\begin{proof}
	We show that both statements are equivalent to $\langle g,g \brangle$ being invertible in $B$. Suppose $\modft_g$ is invertible. Then $E$ is finitely generated and projective as an $A$-module, as we can make the maps $r$ and $s$ from the proof of \cref{Proposition: Fin gen proj iff unital} using $g$ and $\modft_g^{-1}g$. Thus $B$ is unital. As
	\begin{equation*}
	\modft_g f = f \langle g,g \brangle,
	\end{equation*}
	statement (1) is equivalent to $\langle g,g\brangle$ being invertible in $B$. On the other hand,
	\begin{equation}
	\modan_g^B (\modan_g^{B})^* b = \modan_g^B (g\cdot b) = \langle g,g\cdot b\brangle = \langle g,g \brangle b.
	\end{equation}
	Since $\modan_g^B (\modan_g^{B})^* \in \End_B (B)$ and $B$ is an ideal in $\End_B (B)$, statement (2) implies that $B$ is unital and the statement is equivalent to $\langle g,g\brangle$ being invertible in $B$. 
\end{proof}
In Gabor analysis one is often concerned with the regularity of the atoms generating a Gabor frame, see \cref{Subsection: Implications for Gabor Analysis}. In case $g$ is so that $\modft_{g}$ is invertible on all of $E$ with $g \in \E$, and $\B \subset B$ is spectral invariant Banach $*$-subalgebra with the same unit as $B$, the canonical dual atom has the following important property.
\begin{prop}
	\label{Proposition: Canonical dual of regular is regular}
	%Let $g \in \E \subset E$ be so that $\modft_{g}$ is invertible. Then $B$ is unital. Let $\B \subset B$ be spectral invariant with the same unit as $B$. Then the canonical dual $\modft_{g}^{-1} g$ is contained in $\E$, too. 
	Let $E$ be an $A$-$B$-equivalence bimodule, with an $\A$-$\B$-pre-equivalence bimodule $\E \subset E$. Suppose $\B \subset B$ is spectral invariant with the same unit. If $g \in \E$ is such that $\Theta_g : E \to E$ is invertible, then the canonical dual $\Theta_g^{-1} g$ is in $\E$ as well.
\end{prop}
\begin{proof}
For $f \in E$ we have
	\begin{equation*}
	\modft_{g} f = \blangle f,g \rangle g = f \langle g,g\brangle.
	\end{equation*}
	%Since $\modft_{g}$ is invertible, $E$ is a finitely generated projective $A$-module as before, hence $B$ is unital. 
	We deduce that $\langle g,g\brangle$ is invertible in $B$ and
	$$
	\modft_{g}^{-1} g = g  \langle g,g\brangle^{-1}.
	$$
	But as $g \in \E$ we have $\langle g,g \brangle \in \B$. By spectral invariance of $\B$ in $B$ it follows that $ \langle g,g\brangle^{-1} \in \B$. Then, since $\E \B \subset \E$, it follows that
	\begin{equation*}
	\modft_{g}^{-1} g = g  \langle g,g\brangle^{-1} \in \E,
	\end{equation*}
	which is the desired assertion. 
\end{proof}
\subsection{Extending to several generators} 
\label{Subsection: Extending to several generators}
We extend the above theory to several generators. Indeed we will lift the $A$-$B$-equivalence bimodule $E$ to an $M_n (A)$-$M_d (B)$-equivalence bimodule, for $d,n\in \N$, and consider a %new 
type of module frame in this matrix setting. We will see in \cref{Subsection: Implications for Gabor Analysis} that this generalizes $n$-multiwindow $d$-super Gabor frames of \cite{jalu18duality}.  

Note that we will index an $n\times d$-matrix by $(i,j)$, $i \in \Z_n$, $j \in \Z_d$, that is, we start indexing at $0$. The reason for this is that in \cref{Subsection: Implications for Gabor Analysis} we will need to incorporate the groups $\Z_k$, $k \in \N$. Here $\Z_k$ denotes the group $\Z/(k\Z)$.

We will consider $M_{n,d}(E)$ as an $M_n (A)$-$M_d (B)$-bimodule. Define an $M_{n} (A)$-valued inner product on $M_{n,d}(E)$ by
\begin{equation}
\begin{split}
\bbracket-,-]&: M_{n,d}(E) \times M_{n,d}(E) \to M_{n} (A) \\
(f,g)&\mapsto \sum_{k\in \Z_d} \begin{pmatrix}
\blangle f_{0,k} , g_{0,k} \rangle & \blangle f_{0,k} , g_{1,k} \rangle & \ldots & \blangle f_{0,k} , g_{n-1,k} \rangle \\
\blangle f_{1,k} , g_{0,k} \rangle & \blangle f_{1,k} , g_{1,k} \rangle & \ldots & \blangle f_{1,k} , g_{n-1,k} \rangle \\
\vdots & \vdots & \ddots & \vdots \\
\blangle f_{n-1,k} , g_{0,k} \rangle & \blangle f_{n-1,k} , g_{1,k} \rangle & \ldots & \blangle f_{n-1,k} , g_{n-1,k} \rangle
\end{pmatrix}.
\end{split}
\end{equation}
The action of $M_{n} (A)$ on $M_{n,d}(E)$ is defined in the natural way, that is
\begin{equation}
(af)_{i,j} = \sum_{k\in \Z_n} a_{i,k}f_{k,j},
\end{equation}
for $a\in M_n (A)$ and $f \in M_{n,d} (E)$.
Likewise we define an $M_{d}(B)$-valued inner product on $M_{n,d}(E)$ in the following way
\begin{equation}
\begin{split}
[-,-\bracketb&: M_{n,d}(E) \times M_{n,d}(E) \to M_{d}(B) \\
(f,g) &\mapsto \sum_{k\in \Z_n} \begin{pmatrix}
\langle f_{k,0} , g_{k,0} \brangle & \langle f_{k,0} , g_{k,1} \brangle & \ldots & \langle f_{k,0} , g_{k,d-1} \brangle \\
\langle f_{k,1} , g_{k,0} \brangle & \langle f_{k,1} , g_{k,1} \brangle & \ldots & \langle f_{k,1} , g_{k,d-1} \brangle \\
\vdots & \vdots & \ddots & \vdots \\
\langle f_{k,d-1} , g_{k,0} \brangle & \langle f_{k,d-1} , g_{k,1} \brangle & \ldots & \langle f_{k,d-1} , g_{k,d-1} \brangle
\end{pmatrix}.
\end{split}
\end{equation}
The right action of $M_{d}(B)$ on $M_{n,d}(E)$ is defined by
\begin{equation}
(fb)_{i,j} = \sum_{k\in \Z_d} f_{i,k}b_{k,j}
\end{equation}
for $f \in M_{n,d}(E)$ and $b\in M_d (B)$.

With this setup, $M_{n,d}(E)$ becomes an $M_n (A)$-$M_d(B)$-equivalence bimodule. Indeed it is not hard to verify the three conditions of \cref{Definition: Equivalence bimodule}. Verifying conditions ii) and iii) is a matter of verifying the statements in each matrix element using that $E$ is an $A$-$B$-equivalence bimodule. Verifying condition i) is a matter of getting density in each matrix entry by choosing elements of $M_{n,d} (E)$ in the correct way. Namely, if we want to get the elements in place $(i,j)$ in $M_n (A)$, then we may for example pair elements of $M_{n,d} (E)$ with nonzero entry only in place $(i,k)$ with elements of $M_{n,d} (E)$ with nonzero entry only in place $(j,k)$, for some $k \in \Z_d$. The analogous procedure holds for $M_d (B)$. Density then follows by $\overline{\blangle E,E \rangle} = A$ and $\overline{\langle E,E \brangle} = B$. 
In particular, we have for $f,g,h \in M_{n,d} (E)$ that
\begin{equation}
\bbracket f,g] h = f[g,h\bracketb,
\end{equation}
and also
\begin{equation}
\begin{split}
&M_n (A) = \mathbb{K}_{M_d (B)} (M_{n,d}(E)), \\
&M_d (B) = \mathbb{K}_{M_n (A)} (M_{n,d} (E)).
\end{split}
\end{equation}
Also, since the new inner products are defined using the inner products $\blangle -,-\rangle$ and $\langle -,-\brangle$, we see that in case we have Banach $*$-subalgebras $\A \subset A$ and $\B \subset B$, as well as an $\A$-$\B$-subbimodule $\E \subset E$ as above, we get
\begin{equation*} 
\bbracket M_{n,d}(\E),M_{n,d}(\E)] \subset M_{n} (\A), \quad [M_{n,d}(\E),M_{n,d}(\E)\bracketb \subset M_{d}(\B),
\end{equation*}
as well as
\begin{equation*}
M_n (\A) M_{n,d}(\E) \subset M_{n,d}(\E), \quad M_{n,d}(\E)M_d (\B) \subset M_d (\E).
\end{equation*}
We wish to reduce the matrix algebra case to the Gabor bimodule case of \cref{Subsection: The Single Generator Case}, so we need to guarantee that spectral invariance of Banach $*$-subalgebras lifts to matrices. For convenience we include the following result.
\begin{lemma}[\cite{sc92}]
	\label{Lemma: Spectral invariance lifts to matrices}
	If $\B $ is a spectral invariant Banach subalgebra of a Banach algebra $B$, then $M_m (\B)$ is a spectral invariant Banach subalgebra of $M_m (B)$, for all $m \in \N$.
\end{lemma}
For $g \in M_{n,d} (E)$ we define as in \cref{Subsection: The Single Generator Case} the analysis operator
\begin{equation}
\begin{split}
\modan_{g} : M_{n,d}(E) &\to M_n (A) \\
f &\mapsto \bbracket  f,g]
\end{split}
\end{equation}
which has as adjoint the operator
\begin{equation}
\begin{split}
\Phi_{g} : M_n (A) &\to M_{n,d}(E) \\
a&\mapsto ag .
\end{split}
\end{equation}
Using these we also define the frame-like operator $\modan_h^* \modan_g =:\modft_{g,h}: M_{n,d}(E) \to M_{n,d} (E)$ by
\begin{equation}
\modft_{g,h} f = \bbracket f,g ] h, \quad \text{ for $f \in M_{n,d} (E)$,}
\end{equation}
and the frame operator $\modan_g^* \modan_g =:\modft_g : M_{n,d} (E) \to M_{n,d} (E)$ by
\begin{equation}
\modft_g f = \bbracket f,g ] g \quad \text{for $f \in M_{n,d} (E)$.}
\end{equation}
As noted in \cref{Subsection: The Single Generator Case}, $\modft_g$ is a positive operator. 

For simplicity, and since it is the case we will most often consider, suppose in the following that $B$ is unital with a faithful finite trace. There is then an induced (possibly unbounded) trace on $A$ as in \cref{Section: Preliminaries}. We may lift these traces to the matrix algebras. Indeed, there are traces on $M_{n}(A)$ and $M_{d}(B)$ satisfying 
%Assuming the existence of a continuous faithful trace on $B$ and the induced (possibly unbounded) trace on $A$ as in \cref{Section: Preliminaries}, there are associated traces on the matrix algebras. For simplicity, and since it is the case we are most interested in, we assume that $B$ is unital with a faithful trace $\tr_B$. Then there are traces on $M_{n}(A)$ and $M_{d}(B)$ satisfying 
\begin{equation}
\tr_{M_{n}(A)} (\bbracket f,g]) = \tr_{M_{d}(B)}([g,f\bracketb)
\end{equation}
for all $f,g \in M_{n,d}(E)$. They are given by
\begin{equation}
\label{Equation: Relating matrix traces}
\begin{split}
\tr_{M_{n}(A)}(\bbracket f,g]) &= \frac{1}{n} \sum_{i\in \Z_n} \tr_A (\bbracket f,g ]_{i,i}), \\
\tr_{M_{d}(B)}([f,g\bracketb) &= \frac{1}{n}\sum_{i\in \Z_d} \tr_B ([f,g\bracketb_{i,i}).
\end{split}
\end{equation}
The trace on $M_d (B)$ extends to a finite trace on the whole algebra, but the same might not be true for the densely defined trace on $M_n (A)$. It is however true if $A$, and hence also $M_n (A)$, is unital. 
\begin{rmk}
    The normalization on the traces in \eqref{Equation: Relating matrix traces} is so that if $A$ is unital, then $\tr_{M_n (A)}(1_{M_n (A)}) = 1$, that is, $\tr_{M_n (A)}$ is a faithful tracial state. In general $\tr_{M_d (B)}$ will not be a state, even if $M_d (B)$ is unital.
\end{rmk}
The following lemma may be verified by elementary computations.
\begin{lemma}
	\label{Lemma: Induced trace positive}
	Let $B$ be a unital $C^*$-algebra. If $\tr_B$ is a faithful trace on $B$, then the induced mapping $\tr_{M_m (B)}$  on the matrix algebra $M_m (B)$, $m \in \N$, given by
	\begin{equation}
	\tr_{M_m (B)} (b) = \sum_{i\in \Z_m} \tr_B (b_{i,i}),
	\end{equation}
	for $b\in M_m (B)$ is also a faithful trace. 
\end{lemma}
%
%\begin{proof}
%	Let $\mathbf{A}= (a_{i,j}) \in M_m (A)$. Then
%	\begin{equation*}
%	\begin{split}
%	\tr_{M_m (A)} (\mathbf{A}^* \mathbf{A}) = \sum_{i\in \Z_m} \tr_A ((\mathbf{A}^*\mathbf{A})_{i,i}) = \sum_{i\in \Z_m} \tr_A (\sum_{j\in \Z_m} a_{i,j}^* a_{j,i}) = \sum_{i\in \Z_m} \sum_{j\in \Z_m} \tr_A (a_{i,j}^* a_{j,i})\geq 0,
%	\end{split}
%	\end{equation*}
%	since every element $a_{i,j}^* a_{j,i}$, $i,j=1\in \Z_m$, is positive, and $\tr_A$ is a positive linear functional. Since $\tr_A$ is faithful, $\tr_A (a_{i,j}^* a_{j,i}) = 0$ if and only if $a_{i,j}^* a_{j,i} = 0$, if and only if $a_{j,i}=0$. As this holds for all $i,j \in \Z_m$, it follows that $\tr_{M_m (A)}$ is also faithful. Lastly, let $\mathbf{R}= (r_{i,j})\in M_m (A)$. Then
%	\begin{equation*}
%	\begin{split}
%	\tr_{M_m (A)} (\mathbf{A}\mathbf{R}) &= \sum_{i\in \Z_m} \tr_A ( (\mathbf{A}\mathbf{R})_{ii}) = \sum_{i\in \Z_m} \tr_A (\sum_{j\in \Z_m} a_{i,j}r_{j,i}) = \sum_{i\in \Z_m} \sum_{j\in \Z_m} \tr_A (a_{i,j}r_{j,i}) \\
%	&= \sum_{j\in \Z_m} \sum_{i\in \Z_m} \tr_A (r_{j,i}a_{i,j}) = \sum_{j\in \Z_m} \tr_A (\sum_{i\in \Z_m} r_{j,i}a_{i,j}) = \sum_{j\in \Z_m} \tr_A ((\mathbf{R}\mathbf{A})_{jj}) \\
%	&= \tr_{M_m (A)}(\mathbf{R}\mathbf{A}),
%	\end{split}
%	\end{equation*}
%	which shows that $\tr_{M_m (A)}$ is a trace. This finishes the proof.
%\end{proof}
We summarize the preceding discussion in the following proposition which allows us to study the $M_n (A)$-$M_d (B)$-equivalence bimodule $M_{n,d}(E)$ by studying the $A$-$B$-bimodule $E$.
\begin{prop} 
	Let $(A,B,E, \tr_B)$ be a left Gabor bimodule. Then for all $n,d\in\N$, the quadruple
	$$(M_n (A) , M_d (B) , M_{n,d} (E)  , \tr_{M_d (B)})$$ 
	with the above defined actions, inner products, and traces is also a left Gabor bimodule. Furthermore, if $(A,B,E,\tr_A , \tr_B,\A,\B,\E)$ is a left Gabor bimodule with regularity, then for all $n,d\in \N$, the septuple
	\begin{equation*}
	(M_n (A), M_d (B), M_{n,d} (E), \tr_{M_d (B)},M_n (\A), M_d (\B), M_{n,d}(\E)),
	\end{equation*}
	with the above defined actions, inner products, and traces is also a left Gabor bimodule with regularity. The analogous statements hold for right Gabor bimodules. 
\end{prop}
Since our focus is on the description of frames in equivalence bimodules for Morita equivalent $C^*$-algebras, we want to do this now on the matrix algebra level and thus introduce an appropriate notion of module frames for the matrix-valued equivalence bimodules.
\begin{defn}
	\label{Definition: Modular (n,d)-matrix frames}
	Let  $g = (g_{i,j})_{i\in \Z_n , j\in \Z_d} \in M_{n,d}(E)$. We say $g$ generates a \textit{module $(n,d)$-matrix frame for $E$ with respect to $A$} if there exists $h = (h_{i,j})_{i\in \Z_n , j\in \Z_d} \in M_{n,d}(E)$ for which
	\begin{equation}
	\label{Equation: Modular (n,d) matrix frames}
	f_{r,s} = \sum_{k\in \Z_d} \sum_{l\in \Z_n} \blangle f_{r,k} , g_{l,k}\rangle h_{l,s},
	\end{equation}
	holds for all $f = (f_{i,j})_{i\in \Z_n , j\in \Z_d}\in M_{n,d}(E)$, $r \in \Z_n$, and $s \in \Z_d$.
\end{defn}
By definition of the above Hilbert $M_n (A)$-module structure on $M_{n,d}(E)$, we see that $g \in M_{n,d}(E)$ generates a module $(n,d)$-matrix frame for $E$ with respect to $A$ if and only if there is $h \in M_{n,d}(E)$ such that
\begin{equation}
\label{Equation: Single atom reconstruction in matrices}
f = \bbracket f,g ] h
\end{equation}
for all $f \in M_{n,d}(E)$. In other words, $g$ generates a module $(n,d)$-matrix frame for $E$ with respect to $A$ if and only if $g$ generates a single module frame for $M_{n,d}(E)$ with respect to $M_n (A)$. When \eqref{Equation: Single atom reconstruction in matrices} is satisfied $M_{n,d} (E)$ is finitely generated projective as an $M_n (A)$-module, so as before it follows by \cref{Proposition: Fin gen proj iff unital} that $M_d (B)$ is unital. Then $B$ is also unital. By the identity
\begin{equation*}
f = \bbracket f,g ] h = f [g,h\bracketb,
\end{equation*}
we deduce that \eqref{Equation: Single atom reconstruction in matrices} is satisfied if and only if $M_d (B)$ is unital and $[g,h\bracketb = 1_{M_d (B)}$. 
\begin{rmk}
	\label{Remark: All single generator results lift}
	By the above discussion it follows that finding module $(n,d)$-matrix frames for the $A$-$B$-equivalence bimodule $E$
	is the same as finding $g,h \in M_{n,d}(E)$ such that $[g,h\bracketb = 1_{M_d (B)}$. That is, it is the same as finding single module frames for $M_{n,d}(E)$ as an $M_n (A)$-module. By \cref{Lemma: Spectral invariance lifts to matrices}, the corresponding statement is true of finding module frames with regularity. Hence all results of \cref{Subsection: The Single Generator Case} can be carried over to the setup in this section. 
\end{rmk}
Even though all results of \cref{Subsection: The Single Generator Case} lift to the induced matrix algebra setup, we want to discuss explicitly two results relating the lifted traces. We show in \cref{Subsection: Implications for Gabor Analysis} that these two results extend the density theorems of Gabor analysis to Gabor bimodules. Since we in \cref{thm:req-on-n-d-for-mod-frame} talk about left Gabor bimodules and in \cref{Theorem: Requirements for module Riesz} talk about right Gabor bimodules, we will for the sake of avoiding confusion not have any convention on the normalization of traces in the two results.
\begin{thm}
	\label{thm:req-on-n-d-for-mod-frame}
	Let %$E$ be an $A$-$B$-equivalence bimodule, where we a priori do not assume either $A$ or $B$ is unital. If $g \in M_{n,d} (E)$ is such that the condition of \eqref{Equation: Modular (n,d) matrix frames} is satisfied for some $h\in M_{n,d} (E)$, then $B$ is unital. If $A$ is also unital, then 
	$(A,B,E,  \tr_B)$ be a left Gabor bimodule. If, in addition, $A$ is unital and $g \in M_{n,d}(E)$ is such that %the condition of \eqref{Equation: Modular (n,d) matrix frames} is satisfied for some $h\in M_{n,d} (E)$, then
	$\modft_g : E \to E$ is invertible, then
	\begin{equation}
	d\tr_B (1_B) \leq n \tr_A (1_A).
	\end{equation}
\end{thm}
\begin{proof}
	%We have that $B$ is unital and as before we obtain $[g,h\bracketb = [h,g\bracketb= 1_{M_{d}(B)}$. This implies $[g,g\bracketb$ is invertible \todo{need proof!} in $M_{d}(B)$. Then  $\modft_g$ is invertible, and
	The assumption that $\modft_g$ is invertible implies $[g,g\bracketb$ is invertible. Then
	\begin{equation}
	u = \modft_g^{-1} g = g[g,g\bracketb^{-1}
	\end{equation}
	is the canonical dual frame for $M_{n,d} (E)$. We have $[g,u\bracketb = [u,g\bracketb = 1_{M_d (B)}$, and by \cref{Proposition: Identity induce idempotent}, $\bbracket  g,u]$ is a projection in $M_{n} (A)$. Using \cref{Lemma: Induced trace positive} and $\bbracket  g,u] \leq 1_{M_{n}(A)}$ in $M_{n}(A)$, as well as \eqref{Equation: Relating matrix traces}, we get
	\begin{equation*}
	\begin{split}
	d\tr_B (1_B) &= n \cdot \frac{1}{n}\sum_{i=1}^{d} \tr_B (1_B) =  n \tr_{M_{d}(B)} (1_{M_{d}(B)}) = n \tr_{M_{d}(B)} ([u,g\bracketb) \\
	&= n \tr_{M_{n}(A)}(\bbracket  g,u]) \leq n \tr_{M_{n}(A)} (1_{M_{n}(A)}) = n \cdot \frac{1}{n}\sum_{i=1}^{n} \tr_A (1_A) = n\tr_A (1_A).
	\end{split}
	\end{equation*}
\end{proof}
\begin{thm} 
	\label{Theorem: Requirements for module Riesz}
	Let %$E$ be an $A$-$B$-equivalence bimodule for two $C^*$-algebras $A$ and $B$. If $g\in M_{n,d}(E)$ is such that $\modan_g \modan_g^{*}:M_n (A) \to M_n (A)$ is an isomorphism, then $A$ is unital. If in addition $B$ is unital, then
	$(A,B,E,\tr_A)$ be a right Gabor bimodule. If, in addition, $B$ is unital and $g\in M_{n,d}(E)$ is such that $\modan_g \modan_g^{*}:M_n (A) \to M_n (A)$ is an isomorphism, then
	\begin{equation}
	d \tr_B (1_B) \geq n \tr_A (1_A).
	\end{equation}
\end{thm}
\begin{proof}
	The assumptions imply $\bbracket  g,g]^{-1} \in M_n (A)$, so it follows as in \cref{Subsection: The Single Generator Case} that
	\begin{equation*}
	1_{M_n (A)} = \bbracket g,g]^{-1} \bbracket g, g] =\bbracket  \bbracket  g,g]^{-1} g,g],
	\end{equation*}
	and $[\bbracket  g,g]^{-1} g, g\bracketb$ is a projection in $M_{d}(B)$ by \cref{Proposition: Identity induce idempotent}. Since $B$ is unital, then, using \cref{Lemma: Induced trace positive} together with $[\bbracket  g,g]^{-1} g,g\bracketb \leq 1_{M_{d}(B)}$ in $M_{d}(B)$, as well as \eqref{Equation: Relating matrix traces}, we get
	\begin{equation*}
	\begin{split}
	n \tr_A (1_A) &=n \cdot \frac{1}{n} \sum_{i=1}^{n} \tr_A (1_A) = n \tr_{M_{n}(A)} (1_{M_{n}(A)}) = n \tr_{M_{n}(A)} (\bbracket  \bbracket  g,g]^{-1} g,g]) \\
	&= n\tr_{M_{d}(B)} ([g,\bbracket  g,g]^{-1} g\bracketb) \leq n\tr_{M_{d}(B)} (1_{M_{d}(B)}) \\
	&= n \cdot \frac{1}{n}\sum_{i=1}^{d} \tr_B (1_B) =  d\tr_B (1_B).
	\end{split}
	\end{equation*}
\end{proof}
\subsection{From a Gabor bimodule to its localization} 
\label{Subsection: Passing to the Localization}

In \cite{lu09} the existence of multi-window Gabor frames for $L^2 (\R^d)$ with windows in Feichtinger's algebra was proved through considerations on a related Hilbert $C^*$-module. Furthermore, in \cite{lu11} projections in noncommutative tori were constructed from Gabor frames with sufficiently regular windows. Thus being able to pass from an equivalence bimodule $E$ to a localization $H_E$ and back is quite important, and we dedicate this section to results on this procedure. We will interpret this in terms of standard Gabor analysis in \cref{Subsection: Implications for Gabor Analysis}, and we will explain how $L^2 (G)$, for $G$ a second countable LCA group, relates to $H_E$ for specific modules $E$ which arise in the study of twisted group $C^*$-algebras.

We denote by $(-,-)_E$ the inner product on the localization of $E$ in $\tr_A$. Concretely, we have $(f,g)_E = \tr_{A} (\blangle f,g\rangle)$. %We then have the following result.
\begin{prop}
	\label{Proposition: Dual atom in E iff in H_E}
	Let $(A,B,E,\tr_B)$ be a left Gabor bimodule, and let $g \in E$. Then there exists an $h\in E$ such that we have $\blangle f,g\rangle h = f$ for all $f\in E$ if and only if there exist constants $C,D > 0$ such that
	\begin{equation}
	\label{Equation: Localization of modular frame}
	C  (f,f)_E \leq  ( f \langle g, g \brangle ,f )_E \leq D  (f,f )_E
	\end{equation}
	for all $f \in H_E$. In other words, $g$ is a module frame for $E$ if and only if the inequalities in \eqref{Equation: Localization of modular frame} are satisfied for some $C,D > 0$. 
\end{prop}
\begin{proof}
	Suppose first that there is an $h \in E$ such that $\blangle f,g \rangle h = f$ for all $f\in E$. By Morita equivalence this implies 
	\begin{equation*}
	f = \blangle f,g \rangle h = f \langle g,h \brangle
	\end{equation*}
	for all $f \in E$. As before, this implies $1_B = \langle g,h \brangle = \langle h,g \brangle $. Since $\tr_B$ is a positive linear functional %$\tr_A$ sends positive elements of $\blangle E,E\rangle$ to $[0,\infty)$. Hence
	we obtain
	\begin{equation*}
	\begin{split}
	( f, f)_E &= \tr_A (\blangle f,f \rangle) \\
	%&= \tr_A (\blangle f\langle g,h\brangle h,  f \langle g, h\brangle \rangle) \\
	&= \tr_A (\blangle f \langle g,h \brangle \langle h,g \brangle , f \rangle) \\
	&= \tr_A (\blangle f \langle g,h \langle h,g\brangle \brangle, f) \\
	&= \tr_A (\blangle f \langle g,\blangle h,h\rangle g\brangle, f\rangle)	\\
	&= \tr_B  (\langle f, f \langle g,\blangle h,h\rangle g\brangle\brangle) \\
	&\leq \tr_B (\langle f,f \langle g,g\brangle \Vert \blangle h,h \rangle \Vert \brangle) \\
	&= \Vert \blangle h,h \rangle \Vert \tr_B (\langle f,f \langle g,g \brangle \brangle) \\
	&= \Vert \blangle h,h \rangle \Vert \tr_A (\blangle f \langle g,g \brangle, f\rangle ) \\
	&= \Vert \blangle h,h \rangle \Vert (f \langle g,g \brangle , f)_E,
	%( \blangle f,g \rangle h, f )_E = ( \blangle f,g \rangle h \langle h,g\brangle , f  )_E \\
	%&= ( \blangle f,g \rangle \blangle h,h\rangle g , f  )_E 
	%\leq ( \Vert \blangle h,h \rangle \Vert \blangle f,g \rangle g,f )_E \\
	%&= \Vert \blangle h,h \rangle\Vert ( \blangle f,g \rangle g,f )_E ,
	\end{split}
	\end{equation*}
	for all $f \in E$, where we have used \cref{Proposition: Inner product positive operator} to deduce 
	\begin{equation*}
	\langle g, \blangle h,h \rangle g\brangle \leq \Vert \blangle h,h \rangle \Vert \langle g,g \brangle.
	\end{equation*} 
	We then get the lower frame bound with $C= \Vert \blangle h,h \rangle \Vert^{-1}$, that is
	\begin{equation*}
	\frac{1}{\Vert \blangle h,h \rangle \Vert} (f,f)_E \leq  ( f \langle g , g\brangle ,f )_E
	\end{equation*}
	for all $f \in E$. By \cref{Localization norm preserved} all intermediate steps involve operators that extend to bounded operators on $ H_E$, so we may extend by continuity. We get the upper frame bound by use of \cref{Proposition: Inner product positive operator} in the following manner
	\begin{equation*}
	\begin{split}
	(f \langle g,g \brangle,f )_E &= \tr_A (\blangle f \langle g,g \brangle, f\rangle) \\
	&= \tr_A ( \blangle f \langle g,g\brangle^{1/2} , f \langle g,g\brangle^{1/2}\rangle ) \\
	%&= \tr_B ( \langle f \langle g,g\brangle^{1/2} , f \langle g,g\brangle^{1/2}\brangle ) \\
	%&\leq \Vert \langle g,g\brangle^{1/2} \Vert^2 \tr_B (\langle f,f \brangle) \\
	&\leq \| \langle g,g \brangle^{1/2} \|^{2} \tr_A (\blangle f,f \rangle) \\
	&= \Vert \langle g,g\brangle \Vert \tr_A (\blangle f,f \rangle) \\
	&= \Vert \blangle g,g \rangle \Vert (f,f)_E,
	\end{split}
	\end{equation*}
	for all $f \in E$. Once again all intermediate steps involve operators that extend to bounded operators on $H_E$ by \cref{Localization norm preserved}, so we may extend the result to all of $H_E$. Thus we have shown that
	\begin{equation*}
	\frac{1}{\Vert \blangle h,h \rangle \Vert} (f,f )_E \leq (  f \langle g,g \brangle ,f )_E \leq \Vert \blangle g,g \rangle \Vert (f,f)_E \\
	\end{equation*}
	for all $f \in H_E$.
	
	Conversely, suppose there are $C,D >0$ such that
	\begin{equation*}
	C (f,f)_E \leq ( f\langle g,g \brangle,f )_E \leq D ( f,f )_E
	\end{equation*}
	for all $f \in H_E$. We wish to show that this implies there exist $h \in E$ such that  $\blangle f,g \rangle h = f$ for all $f \in E$. The assumption implies that $f \mapsto f \langle g,g \brangle$ is a positive, invertible operator on $H_E$. By \cref{Proposition: C*-subalgebra inverse closed} it follows that $\langle g,g \brangle$ is invertible in $B$. 
	Thus $f \mapsto f \langle g,g \brangle$ is a positive, invertible operator on $E$ as well. Hence the operator
	\begin{equation*}
	\begin{split}
	\modft_g &: E \to E \\
	f &\mapsto \blangle f,g \rangle g = f 
	\langle g,g\brangle
	\end{split}
	\end{equation*}
	is invertible with inverse
	\begin{equation*}
	\modft_g^{-1} f = f \langle g,g\brangle^{-1}.
	\end{equation*}
	Define $h := \modft_g^{-1} g$, and let $f \in E$ be arbitrary. Then we have
	\begin{equation*}
	\begin{split}
	\blangle f,g \rangle h &= \blangle f,g \rangle \modft_g^{-1} g = \modft_g^{-1} (\blangle f,g \rangle g ) = \modft_g^{-1} \modft_g f = f,
	\end{split}
	\end{equation*}
	from which the result follows.
\end{proof}
%Note that as $\tr_{B} (\langle f,g \brangle) = \tr_{A} (\blangle g,f\rangle )$ for all $f,g \in E$, the localization of $E$ in $\tr_B$ is the same as the localization of $E$ in $\tr_A$, that is, both localizations are the Hilbert space $H_E$. 

We are interested in module frames and module Riesz sequences, and their relationship to frames and Riesz sequences in Gabor analysis for LCA groups.
To get results on Riesz sequences in \cref{Subsection: Implications for Gabor Analysis} we need a module version of Riesz sequences which, when localized, yields the Riesz sequences we know from Gabor analysis. For this we let $A$ be unital with a faithful trace $\tr_A$, and we need to localize $A$ as a Hilbert $A$-module in the trace $\tr_A$. We let $(a_1, a_2)_A := \tr_A (a_1 a_2^*)$. The completion of $A$ in this inner product will be denoted $H_A$, and the action of $A$ on $H_A$ is the continuous extension of the multiplication action (from the right) of $A$ on itself. %Then we have the following result.
\begin{prop} 
	\label{Proposition: Modular Riesz sequences}
	Let $(A,B,E,\tr_A)$ be a right Gabor bimodule, and let $g \in E$. Then $\modan_g \modan_g^*: A \to A$ is an isomorphism if and only if there exist $C,D > 0$ such that for all $a \in  A$ it holds that
	\begin{equation}
	\label{Equation: Riesz trace inequality}
	C (a,a)_A \leq ( ag,ag )_E \leq D  (a,a)_A. 
	\end{equation}
	\begin{proof}
		First suppose $\modan_g \modan_g^{*} :A \to A$ is an isomorphism. Then, since by \cref{Proposition: Middle positivity}
		\begin{equation*}
		\blangle ag,ag \rangle = a \blangle g,g \rangle a^* \leq \Vert \blangle g,g \rangle \Vert aa^*,
		\end{equation*}
		we may deduce
		\begin{equation*}
		(ag,ag)_A = \tr_A (\blangle ag ,ag\rangle)  \leq \Vert \blangle g,g \rangle \Vert \tr_A (aa^*) = \Vert \blangle g,g \rangle \Vert (a,a)_A.
		\end{equation*}
		Hence in \eqref{Equation: Riesz trace inequality} we may set $D = \Vert \blangle g,g \rangle \Vert$. Since $\modan_g \modan_g^* : A \to A$ is an isomorphism and $\modan_g \modan_g^{*} a = a\blangle g,g\rangle $, it follows that there is $\blangle g,g\rangle^{-1} \in A$. Then 
		\begin{equation*}
		\begin{split}
		(a,a)_A &= \tr_A(aa^* )_A\\
		%&= (a\blangle g,g\rangle^{1/2}\blangle g,g\rangle^{-1/2} , a \blangle g,g\rangle^{1/2} \blangle g,g\rangle^{-1/2} )_A \\
		&= \tr_A (a\blangle g,g\rangle^{1/2}\blangle g,g\rangle^{-1}\blangle g,g\rangle^{1/2} a^* ) \\
		&\leq \Vert \blangle g,g\rangle^{-1}\Vert \tr_A (a \blangle g,g \rangle a^*) \\
		&= \Vert \blangle g,g\rangle^{-1}\Vert \tr_A (\blangle ag,ag \rangle ) \\
		&= \Vert \blangle g,g\rangle^{-1}\Vert (ag,ag)_E,
		\end{split}
		\end{equation*}
		which implies that we may set $C = \Vert \blangle g,g \rangle^{-1} \Vert^{-1}$ in \eqref{Equation: Riesz trace inequality}. All intermediate steps extend to $H_A$ by \cref{Localization norm preserved}.
		
		Suppose now that \eqref{Equation: Riesz trace inequality} is satisfied.
		%Extending the $A$-action to $H_A$ by continuity, t
		The lower inequality in \eqref{Equation: Riesz trace inequality} tells us that for all $a \in A$,
		\begin{equation*}
		\begin{split}
		(a(\blangle g,g \rangle  - C) , a )_A &= \tr_A (a (\blangle g,g \rangle - C)a^*) \\
		&= \tr_A ( a\blangle g,g \rangle a^*) - C\tr_A (aa^*) \\
		&= \tr_A (\blangle ag,ag \rangle ) - C \tr_A (aa^*) \\
		&= (ag,ag)_E - C (a,a)_A \geq 0.
		\end{split}
		\end{equation*}
		Note that we need the upper inequality of \eqref{Equation: Riesz trace inequality} to extend all intermediate steps to $H_A$ via \cref{Localization norm preserved}. It follows that $\blangle g,g\rangle$ is a positive invertible operator on $H_A \supset A$. By \cref{Proposition: C*-subalgebra inverse closed} it follows that $\blangle g,g\rangle$ is invertible in $A$. Then, since 
		\begin{equation*}
		\modan_g \modan_g^{*} a = a\blangle g,g \rangle ,
		\end{equation*}
		it follows that $\modan_g \modan_g^{*}:A \to A$ is an isomorphism. 
	\end{proof}
\end{prop}

	Both \cref{Proposition: Dual atom in E iff in H_E} and \cref{Proposition: Modular Riesz sequences} were proved for Gabor bimodules, so by \cref{Remark: All single generator results lift} the results lift to the corresponding matrix setting of \cref{Subsection: Extending to several generators}.

\begin{rmk}
	\label{Remark: All module atoms are Bessel vectors}
	
	Note that in the proofs of the two preceding results the upper bounds in \eqref{Equation: Localization of modular frame} and \eqref{Equation: Riesz trace inequality} were both satisfied with $D = \Vert \blangle g,g\rangle \Vert$. We will see in \cref{Subsection: Implications for Gabor Analysis} that in the Gabor analysis setting, this means that all atoms coming from the Hilbert $C^*$-module are Bessel vectors for the localized frame system. 
\end{rmk}
\begin{rmk}
	The two preceding results actually have shorter proofs using \cref{Proposition: C*-subalgebra inverse closed} in a more direct way, but these proofs would not give us values for $C$ and $D$ in \eqref{Equation: Localization of modular frame} and \eqref{Equation: Riesz trace inequality}, only the existence. The values of the constants are of interest on their own, see \cref{Subsection: Implications for Gabor Analysis}.
\end{rmk}
For use in \cref{Subsection: Implications for Gabor Analysis}, we introduce the following notion.
\begin{defn}
	\label{Definition: Modular Riesz sequence}
	Let $(A,B,E,\tr_A)$ be a right Gabor bimodule, and let $g \in E$. If $\modan_g \modan_g^{*}:A \to A$ is an isomorphism, we say $g$ generates a \textit{module Riesz sequence for $E$ with respect to $A$.} If $h \in M_{n,d}(E)$ generates a module Riesz sequence for $M_{n,d}(E)$ with respect to $M_n (A)$, we will also say that $h$ generates a \textit{module $(n,d)$-matrix Riesz sequence for $E$ with respect to $A$}.
\end{defn}

\section{Applications to Gabor analysis}%\label{Section: Applications} 
%\subsection{Gabor Analysis for LCA groups}
\label{Subsection: Implications for Gabor Analysis}
In this section we show how the above results reproduce some of the core results of Gabor analysis for LCA groups. We will see how some of the cornerstones of Gabor analysis on LCA groups are trivial consequences of the above framework. Of particular interest is the reproduction of some of the main results of \cite{jalu18duality} on 
$n$-multiwindow $d$-super Gabor frames with windows in the Feichtinger algebra. Indeed we will show the corresponding results for localized module $(n,d)$-matrix frames, which generalize $n$-multiwindow $d$-super Gabor frames.

To present the results we will need to explain how time frequency analysis on LCA groups relates to Morita equivalence of twisted group $C^*$-algebras. In the interest of brevity, we refer the reader to \cite{jalu18duality} for a more in-depth treatment of time frequency analysis and its relation to twisted group $C^*$-algebras, and to \cite{ja18} for a survey on the Feichtinger algebra.

Throughout this section, we fix a second countable LCA group $G$ and let $\widehat{G}$ be its dual group. We fix a Haar measure $\mu_G$ on $G$ and normalize the Haar measure $\mu_{\widehat{G}}$ on $\widehat{G}$ such that the Plancherel theorem holds. By $\La$ we denote a closed subgroup of the time-frequency plane $G \times \widehat{G}$. The induced topologies and group multiplications on $\La$ and $(G\times \widehat{G})/\La$ turn them into LCA groups as well, and we may equip them with their respective Haar measures. Having fixed the Haar measures on $G, \widehat{G}$, and $\La$, we will assume $(G \times \widehat{G})/\La$ is equipped with the unique Haar measure such that Weil's formula holds, that is, such that for all $f \in L^1 (G\times \widehat{G})$ we have
\begin{equation*}
\int_{G\times \widehat{G}} f(\xi) d\mu_{G\times \widehat{G}} = \int_{(G\times \widehat{G})/\La} \int_{\La} f(\xi + \la ) d\mu_{\La}(\la) d\mu_{(G\times \widehat{G})/\La}(\Dot{\xi}), \quad \text{$\Dot{\xi} = \xi + \La$}.
\end{equation*}
In this setting we can define the \textit{size of $\La$} by
\begin{equation*}
s(\La) := \int_{(G \times \widehat{G})/\La} 1 d\mu_{(G \times \widehat{G})/\La}.
\end{equation*}
Note that $s (\La)$ is finite if and only if $\La$ is cocompact in $G \times \widehat{G}$. 

For any $x \in G$ and $\omega \in \widehat{G}$ we define the translation operator (or time shift) $T_x$ by
\begin{equation*}
T_x f(t) = f(t-x), \quad t\in G,
\end{equation*}
and the modulation operator (or frequency shift) $E_{\omega}$ by
\begin{equation*}
E_{\omega}f(t) = \omega (t) f(t), \quad t \in G.
\end{equation*}
These operators are unitary on $L^2 (G)$, and satisfy the commutation relation
\begin{equation}
E_{\omega} T_x = \omega (x) T_x E_{\omega}.
\end{equation}
For any $\xi = (x,\omega) \in G \times \widehat{G}$ we may then define the time-frequency shift operator
\begin{equation}
\pi (\xi) = \pi (x ,\omega ) = E_{\omega}T_x.
\end{equation}
We define the $2$-cocycle
\begin{equation}
\begin{split}
c: (G \times \widehat{G}) \times (G \times \widehat{G}) &\to \mathbb{T} \\
(\xi_1 , \xi_2 ) &\mapsto \overline{\omega_2 (x_1)}
\end{split}
\end{equation}
for $\xi_1 = (x_1 , \omega_1), \xi_2 = (x_2 , \omega_2) \in G \times \widehat{G}$. Note then that
\begin{equation}
\pi (\xi_1) \pi(\xi_2) = c(\xi_1 , \xi_2) \pi(\xi_1 + \xi_2),
\end{equation}
and
\begin{equation}
\pi (\xi_1)\pi(\xi_2) = c(\xi_1,\xi_2)\overline{c (\xi_2 , \xi_1)}\pi(\xi_2)\pi(\xi_1) = c_s (\xi_1, \xi_2)\pi(\xi_2)\pi(\xi_1),
\end{equation}
where we have introduced the symplectic cocycle $c_s$ by 
\begin{equation}
\begin{split}
c_s : (G \times \widehat{G}) \times (G \times \widehat{G}) &\to \mathbb{T} \\
(\xi_1 , \xi_2 ) &\mapsto \overline{\omega_2 (x_1)} \omega_1 (x_2).
\end{split}
\end{equation}
We also remark that
\begin{equation}
\pi (\xi)^* = c (\xi,\xi)\pi (-\xi)
\end{equation}
for all $\xi \in G \times \widehat{G}$.

Using the symplectic cocycle $c_s$ we define for a closed subgroup $\La \subset G \times \widehat{G}$ the \textit{adjoint subgroup} $\La^{\circ}$ by
\begin{equation}
\La^{\circ} := \{ \xi \in G \times \widehat{G} \mid c_s (\xi , \la) = 1 \quad \forall \la\in \La\}.
\end{equation}  
Then $(\La^{\circ} )^{\circ} = \La$ and $\widehat{\La^{\circ}} \cong (G\times \widehat{G})/\La$, see for example \cite{JaLe16}. Note that $\La$ is cocompact if and only if $\La^{\circ}$ is discrete. With these identifications we put on $\La^{\circ}$ the Haar measure such that the Plancherel theorem holds with respect to $\La^{\circ}$ and $(G\times \widehat{G})/\La$. 

We define the \textit{short time Fourier transform} with respect to $g \in L^2 (G)$ as the operator
\begin{equation}
\begin{split}
V_g : L^2 (G) &\to L^2 (G\times \widehat{G})\\
V_g f(\xi) &= \langle f, \pi (\xi) g \rangle, 
\end{split}
\end{equation}
for $\xi \in G \times \widehat{G}$. The \textit{Feichtinger algebra} $S_0 (G)$ is then defined by
\begin{equation}
S_0 (G) := \{f \in L^2 (G) \mid V_f f \in L^1 (G \times \widehat{G})\}. 
\end{equation}
A norm on $S_0 (G)$ is given by
\begin{equation}
\Vert f \Vert_{S_0 (G)} := \Vert V_g f\Vert_{L^1 (G \times \widehat{G})}, \quad \text{for some $g \in S_0 (G)\setminus \{0\}$.}
\end{equation}
It is a nontrivial fact that all elements of $S_0 (G)\setminus \{0\}$ determine equivalent norms on $S_0 (G)$. In case $G$ is discrete it is known that $S_0 (G) = \ell^1 (G)$ with equivalent norms. Furthermore, $S_0 (G)$ consists of continuous functions and is dense in both $L^1 (G)$ and $L^2 (G)$.
%Feichtinger's algebra is an extensively studied algebra, and as mentioned earlier we refer the reader to \cite{ja18} for an in-depth treatment of Feichtinger's algebra for LCA groups.

For two functions $F_1,F_2$ over $\La \subset G\times \widehat{G}$ we define the twisted convolution by
\begin{equation}
F_1 \natural F_2 (\la) := \int_{\La} F_1 (\la') F_2 (\la - \la') c(\la' , \la - \la') \dif \la',
\end{equation}
and the twisted involution $F_1^* (\la) := c(\la, \la) \overline{F_1(-\la)}$. In \cite{jalu18duality} it was shown that $S_0 (\La)$ is a Banach $*$-$D$-algebra for some $D>0$ when equipped with twisted convolution, and indeed it is possible to choose an equivalent norm on $S_0 (G)$ such that it becomes a Banach $*$-algebra. We denote the resulting Banach $*$-algebra by $S_0 (\La,c)$. Using this we may then define two Banach $*$-algebras
\begin{equation}
\label{Equation: Definition left/right Feichtinger algebras}
\begin{split}
&\A := \{ \mathbf{a} \in \mathbb{B}(L^2 (G)) \mid \mathbf{a} = \int_{\La} a (\la) \pi (\la) \dif \la , a\in S_0 (\La) \}, \\
&\B := \{ \mathbf{b}\in \mathbb{B}(L^2 (G)) \mid \mathbf{b} = \int_{\La^{\circ}} b(\la^{\circ}) \pi (\la^{\circ})^* \dif \la^{\circ} , b \in S_0 (\La^{\circ}) \}. 
\end{split}
\end{equation}
Note that $\A \cong S_0 (\La,c)$ and $\B \cong S_0 (\Lao , \overline{c})$ via the natural maps. We will use these identifications without mention in the sequel.

The following was proved in \cite{jalu18duality}.
\begin{lemma}
	Suppose $\La \subset G\times \widehat{G}$ is a closed subgroup. Then $\xi \to \pi (\xi)$ is a faithful unitary $c$-projective representation of $\La$. As a result, the integrated representation is a non-degenerate $*$-representation of $S_0 (\La , c)$. 
\end{lemma}
We may then obtain the minimal universal enveloping algebra $C_r^* (\La , c)$ of $S_0 (\La , c)$ through the integrated representation of $S_0 (\La , c)$ on $L^2 (G)$, that is, the representation
\begin{equation}
\mathbf{a}\cdot f = \int_{\La} a(\la) \pi (\la) f d\la,
\end{equation}
for $\mathbf{a}\in S_0 (\La, c)$ and $f \in L^2 (G)$. As $\La$ is abelian, hence amenable, the minimal and maximal enveloping algebras coincide, so we write $C^* (\La, c)$ for the universal enveloping algebra of $S_0 (\La, c)$. We do the same for $S_0 (\La^{\circ}, \overline{c})$, %Note the conjugated cocycle $\overline{c}$. This is due to the adjoint operation in the definition of $\B$ in \eqref{Equation: Definition left/right Feichtinger algebras}. 
and denote its universal enveloping $C^*$-algebra by $C^* (\La^{\circ} , \overline{c})$.  
\begin{thm}[\cite{ri88}]
	The twisted group $C^*$-algebras $C^* (\La, c)$ and $C^* (\La^{\circ} ,  \overline{c})$ are Morita equivalent. 
\end{thm}
Indeed, $S_0 (G)$ becomes a pre-equivalence bimodule between $\A$ and $\B$ as in \cref{Definition: Equivalence bimodule} when equipped with the inner products
\begin{equation}
\label{Equation: Left Gabor inner product}
\blangle f,g \rangle = \int_{\La} \langle f, \pi (\la) g \rangle \pi (\la) d\la
\end{equation}
and 
\begin{equation}
\label{Equation: Right Gabor inner product}
\langle f,g \brangle = \int_{\La^{\circ}} \langle g, \pi (\la^{\circ})^* f \rangle \pi (\la^{\circ})^* d\la^{\circ},
\end{equation}
and the actions
\begin{equation}
\label{Equation: Left Gabor action}
\mathbf{a} \cdot f = \int_{\La} a(\la )\pi (\la )f d\la
\end{equation}
and
\begin{equation}
\label{Equation: Right Gabor action}
f\cdot \mathbf{b} = \int_{\La^{\circ}} b(\la^{\circ}) \pi(\la^{\circ})^*f d\la^{\circ},
\end{equation}
with $\mathbf{a}\in \A$, $\mathbf{b}\in \B$, and $f,g\in S_0 (G)$. That these are well-defined was noted in Section 3 of \cite{jalu18duality}.
In the remainder of the section we denote by $A$ the $C^*$-completion of $\A$, $B$ the $C^*$-completion of $\B$, $\E = S_0 (G)$, and by $E$ the Hilbert $C^*$-module completion of $\E$. Hilbert $C^*$-modules $E$ as in this setting are called \textit{Heisenberg modules}. 
\begin{rmk}
	The fact that we get the same twisted group $C^*$-algebras by using $S_0 (\La ,c)$ as we get when using the more traditional approach with $L^1 (\La,c)$ was noted in \cite{AuEn19}.
\end{rmk}

Since $S_0$-functions are continuous, there are also well-defined canonical faithful traces on $\A$ and $\B$ given by
\begin{equation}
\begin{split}
\tr_{\A}: \A &\to \C \\
\mathbf{a} &\mapsto a(0),
\end{split}
\end{equation}
and
\begin{equation}
\begin{split}
\tr_{\B}: \B &\to \C \\
\mathbf{b} &\mapsto b(0).
\end{split}
\end{equation}
In general, these traces do not extend to $A$ and $B$, but we will nonetheless denote them by $\tr_A$ and $\tr_B$. These are indeed related as in \eqref{Equation: Trace relation}, which can be seen from
\begin{equation}
\label{Equation: Localization inner product}
\tr_A (\blangle f,g\rangle ) = \blangle f,g\rangle (0) = \langle f,g \rangle_{L^2 (G)} = \langle g,f \brangle (0) = \tr_B (\langle g,f \brangle ).
\end{equation}

In our discussion the following two results are crucial. The first follows immediately by \cite{milnes71}, and the second is a consequence of \cite{grle04}.
\begin{lemma}
	\label{Lemma: Group C*-alg unital iff discrete}
	$C^* (\La, c)$ is unital if and only if $\La$ is discrete.
\end{lemma}
\begin{prop}
	\label{Proposition: Spectral invariance when discrete}
	For a discrete subgroup $\La$ in $G\times\widehat{G}$ the involutive Banach algebra $S_0 (\La,c)$ is spectral invariant in $C^* (\La, c)$. 
\end{prop}
%Recall that we have $S_0 (\La,c) = \ell^1 (\La,c)$ when $\La$ is discrete.
\begin{rmk}
	Although the traces $\tr_A$ and $\tr_B$ do not in general extend to the algebras $A$ and $B$, we can guarantee they extend in one case. Namely, $\tr_A$ extends to all of $A$ if $A$ is unital, which is equivalent to $\La$ being discrete. The same is of course true for $B$ and $\tr_B$, with the discreteness condition on $\Lao$. This is due to the fact that the trace given by evaluation in the identity extends to twisted group $C^*$-algebras when the underlying group is discrete \cite[p. 951]{BeOm18}.
	
	The case of $\La$ or $\Lao$ being discrete is the case we will almost exclusively restrict to after \cref{Proposition: When is S_0 (G) fin gen proj module}.
\end{rmk}
The following is now an immediate consequence.
\begin{prop}
	Let $\La$ be cocompact, which implies $\Lao$ is discrete. Then under the above conditions on $A,B,E, \tr_B$ the quadruple $(A,B,E, \tr_B)$ is a left Gabor bimodule. In addition, the septuple $(A,B,E,\tr_B,\A,\B,\E)$ is a left Gabor bimodule with regularity.

	If $\Lao$ is cocompact and thus $\La$ is discrete, then we obtain a right Gabor bimodule with regularity analogously.
\end{prop}
%\begin{rmk}
%	By \cref{Lemma: Group C*-alg unital iff discrete} and \cref{Proposition: Spectral invariance when discrete} it follows that with $A,B,E,\A,\B,\E,\tr_A$, and $\tr_B$ as above, the octuple $(A,B,E,\tr_A,\tr_B,\A,\B,\E)$ is a left Gabor bimodule with regularity whenever $\Lao$ is discrete, see \cref{Definition: Gabor bimodule}. This is equivalent to $\La \subset G \times \widehat{G}$ being cocompact.
%\end{rmk}
We may then reprove Theorem 3.9 of \cite{jalu18duality} in this framework. 
\begin{prop}
	\label{Proposition: When is S_0 (G) fin gen proj module}
	Let $\La \subset G \times \widehat{G}$ be a closed subgroup. %, and let $E$ be the completion of $S_0 (G)$ as a $C^* (\La, c)$-module. 
	Then $E$ is a finitely generated projective $A$-module if and only if $\La \subset G \times \widehat{G}$ is a cocompact subgroup. Also, $\E$ is a finitely generated projective $\A$-module if and only if $\La \subset G \times \widehat{G}$ is cocompact.
\end{prop}
\begin{proof} 
	$E$ is finitely generated and projective over $A$ if and only if $\K_{A} (E) = \End_{A} (E)$. As $E$ is an $A$-$B$-equivalence bimodule, this is equivalent to $B$ being unital by \cref{Proposition: Fin gen proj iff unital}. $B$ is unital if and only if $\La^{\circ}$ is discrete by \cref{Lemma: Group C*-alg unital iff discrete}, so equivalently
	\begin{equation}
	\widehat{\La^{\circ}} \cong (G \times \widehat{G})/\La
	\end{equation}
	is compact, that is, $\La$ is cocompact in $G \times \widehat{G}$. 
	
	Now, if $\La$ is cocompact, then $\B$ is unital, so we are in the situation of \cref{Proposition: Fin gen proj passes to subalg when spectral invariant} by \cref{Proposition: Spectral invariance when discrete}. Hence by the first part of this proposition it follows that $\E$ is a finitely generated projective $\A$-module.
	
	Conversely, suppose $\E$ is a finitely generated projective $\A$-module. Then $\E \cong \A^n p$ isometrically for some $n \in \N$ and some $p \in M_n (\A)$. Passing to the completions we obtain $E \cong A^n p$, so $E$ is a finitely generated projective $A$-module. By the first part of this proposition it follows that $\La$ is cocompact. 
	%With $\E = S_0 (G)$ with $\A = S_0 (\La,c)$ and $\B = S_0 (\La^{\circ}, \overline{c})$, we are in the situation of \cref{Proposition: Fin gen proj passes to subalg when spectral invariant} by \cref{Proposition: Spectral invariance when discrete}. It is always true %\todo{is it?Franz: If we take the module norm and have spectral invariant subalgebras. Hence I would just state it for our setup and not claim that it is true in general} 
	%that a projective Banach inner product module over a Banach $*$-algebra completes to a projective Hilbert $C^*$-module over the corresponding $C^*$-completion of the Banach $*$-algebra. Hence it follows by \cref{Proposition: Fin gen proj passes to subalg when spectral invariant} that $S_0 (G)$ is a finitely generated projective $S_0 (\La, c)$-module if and only if $\La \subset G \times \widehat{G}$ is cocompact. 
\end{proof}
\begin{rmk}
	\cref{Proposition: When is S_0 (G) fin gen proj module} shows that we can only have finite module frames for $E$ as an $A$-module if $\La$ is cocompact in $\tfp{G}$. Since we wish to study the relationship between finite module frames and Gabor frames this is the case we care most about in the sequel. %Indeed, unless otherwise specified this will be a standing assumption.%We will see below that this means that the desired Gabor frames can only exist if $\La$ is cocompact in $\tfp{G}$. %Hence this is the case we care the most about in the sequel.
\end{rmk}
To get results on Gabor frames for $L^2 (G)$ with windows in $E$ from the above setup, we will need to localize certain subsets of the $C^*$-algebras $A$ and $B$, as well as the Morita equivalence bimodule $E$, just as explained in \cref{Section: Preliminaries}. For simplicity, let $\La$ be cocompact in $\tfp{G}$ from now on, unless otherwise specified. Then $\Lao$ is discrete and $\tr_B$ is defined on all of $B$. The localization of $B$ in $\tr_B$ is induced by the inner product $(-,-)_B$ given by
\begin{equation*}
(b_1,b_2)_B := \tr_B (b_1^* b_2).
\end{equation*}
Since $\B$ is dense in $B$ and $\tr_B$ is continuous, it follows that their localizations in $\tr_B$ are the same. For $b_1,b_2 \in \B$ we then have 
\begin{equation*}
\begin{split}
(b_1,b_2)_B &= \tr_B (b_1^* b_2) \\
&= \tr_B(\sum_{\lao \in \Lao} \overline{b_1 (\lao)}\pi (\lao)^* \sum_{\xi\in \Lao} b_2 (\xi) \pi (\xi))\\
&= \tr_B (\sum_{\lao \in \Lao}\sum_{\xi \in \Lao} \overline{b_1(\lao)}b_2 (\xi) c(\lao,\lao) \pi (-\lao) \pi (\xi)   )\\
&= \tr_B ( \sum_{\lao \in \Lao}\sum_{\xi\in \Lao} \overline{b_1 (\lao)} b_2 (\xi) c(\lao,\lao) c(-\lao,\xi ) \pi(-\lao + \xi)) \\
&= \tr_B ( \sum_{\lao \in \Lao}\sum_{\xi\in \Lao} \overline{b_1 (\lao + \xi)} b_2 (\xi) c (\lao + \xi, \lao + \xi) c(-\lao - \xi, \xi) \pi(-\lao)   ) \\
&= \sum_{\xi \in \Lao} \overline{b_1 (\xi)} b_2 (\xi) c(\xi,\xi) c(-\xi,\xi) \\
&= \sum_{\xi \in \Lao} \overline{b_1 (\xi)} b_2 (\xi) \\
&= \langle b_1 , b_2 \rangle_{\ell^2 (\Lao)}.
\end{split}
\end{equation*}
As $\B = S_0 (\Lao,\overline{c}) = \ell^1 (\Lao,\overline{c})$ is dense in $\ell^2 (\Lao)$, we may identify the localization $H_B$ of $B$ with $\ell^2 (\Lao)$. By \cite[Proposition 3.2]{AuEn19} we also obtain that the localization of $E$ in $\tr_B$ is $L^2 (G)$. Note that this is the same as the localization of $E$ in $\tr_A$ by construction, and that there is an action of $A$ on $L^2 (G)$ by extending the action of $A$ on $E$.

It is slightly more tricky to localize subsets of $A$. Indeed, it is not in general possible as the trace might not be defined everywhere. However, even if $A$ is not unital we may localize the algebraic ideal $\blangle E,E \rangle \subset A$ in the trace $\tr_A$. Indeed, by \cite[Theorem 3.5]{AuEn19}, elements of $E$ are such that whenever $g \in E$ and $f \in L^2 (G)$, then $\{ \langle f, \pi (\la)g\rangle \}_{\la \in \La} \in L^2 (\La)$. This is the property of being a Bessel vector, which we will discuss in more detail below. Hence for any $f,g \in E$, we may identify $\blangle f,g \rangle \in A$ with $(\langle f, \pi (\la) g\rangle)_{\la \in \La}$ in $L^2 (\La)$ by doing the analogous procedure with $\tr_A$ as for $\tr_B$ above. %Since convergence in $E$-norm implies convergence in $L^2$-norm by \cite[Insert prop]{AuEn19}, it follows by \cite[Lemma 3.1.1, Corollary 3.2.2]{gr01} that we can localize elements of $\blangle E,E\rangle$ by a limit procedure. To do this, note that by the results just cited, for any $f,g \in E$ we have
%\begin{equation*}
%	\tr_A (\blangle f,g \rangle ) = \lim_{n\to \infty} \tr_A(\blangle f_n ,g_n \rangle)
%\end{equation*}
%for any sequences $(f_n)_n, (g_n)_n \subset S_0 (G)$ with $\lim_{n\to \infty} f_n = f$, $\lim_{n\to \infty} g_n = g$ in $E$-norm. We may then define the localization of $\blangle f,g \rangle$ as the $L^2 (\La)$-norm limit of the localization of $\blangle f_n ,g_n \rangle$.

We may do the same for the matrix algebras and matrix modules considered in \cref{Subsection: Extending to several generators}. Note that $\bbracket M_{n,d} (E), M_{n,d}(E) ] = M_{n,d}(\blangle E,E\rangle)$. Adapting the setting of twisted group $C^*$-algebras and Heisenberg modules above to the matrix algebra setting of \cref{Subsection: Extending to several generators} we see that we obtain the following identifications
\begin{equation}
\begin{split}
&H_{M_d (B)} = \ell^2 (\Lao \times \Z_d \times \Z_d) \\
&H_{M_n (\blangle E,E \rangle)} = L^2 (\La \times \Z_n \times \Z_n) \\
&H_{M_{n,d} (E)} = L^2 (G \times \Z_n \times \Z_d).
\end{split}
\end{equation}
\begin{rmk}
	Should $\Lao$ be cocompact and therefore $\La$ discrete, we do the obvious changes. Also if both $\La$ and $\Lao$ are discrete, that is, they are both lattices, then we may localize all of $M_n (A)$ and all of $M_d (B)$.
\end{rmk}

%Since $\tr_{M_n (A)}$ and $\tr_{M_d (B)}$ are continuous, coupled traces, all adjointable module operators descend to bounded Hilbert space operators on the localizations. In other words, for $X,Y \in \{M_n (A), M_d (B), M_{n,d}(E)\}$ and $T: X \to Y$ an adjointable module operator, the natural map $\rho (T)$ in the following diagram exists, and the diagram commutes
%\begin{center}
%	\begin{tikzpicture}[scale=1.0]
%	\node (1) at (0,2) {$X$};
%	\node (2) at (2,2) {$Y$};
%	\node (3) at (0,0) {$H_X$};
%	\node (4) at (2,0) {$H_Y$};
%	\path[->,font=\scriptsize]
%	(1) edge[auto] node[auto] {$T$} (2)
%	(1) edge node[auto,swap] {$\rho_X$} (3)
%	(2) edge node[auto] {$\rho_Y$} (4)
%	(3) edge[densely dashed] node[auto] {$\rho (T)$} (4);
%	\end{tikzpicture}
%\end{center}
%Here $\rho_X, \rho_Y$ are the corresponding localization maps. The map $\rho(T)$ is a bounded linear operator between the Hilbert spaces.

We can finally treat the analogs of Gabor frames in our framework.
In what follows we will consider a novel type of Gabor frames. To ease notation we will for $f \in L^2 (G \times \Z_n \times \Z_d)$ write $f_{i,j}$ instead of $f(\cdot ,i,j)$, and the same for elements of $L^2 (\La \times \Z_n \times \Z_n)$ and $L^2 (\Lao \times \Z_d \times \Z_d)$.
\begin{defn}
	\label{Definition: nd-matrix frames}
	Let $\La$ be a closed subgroup of $G \times \widehat{G}$. For $g \in L^2 (G \times \Z_n \times \Z_d)$ we define the coefficient operator $C_g$ by
	\begin{equation}
	\label{Equation: Gabor coefficient operator}
	\begin{split}
	C_g : L^2 (G \times \Z_n \times \Z_d) &\to L^2 (\La \times \Z_n \times \Z_n) \\
	C_g (f) &= \{\sum_{m\in \Z_d}\langle f_{k,m} , \pi(\la)g_{l,m}\rangle \}_{\la \in \La, k,l\in \Z_n}
	\end{split}
	\end{equation}
	and the synthesis operator $D_g$ by
	\begin{equation}
	\label{Equation: Gabor synthesis operator}
	\begin{split}
	D_g : L^2 (\La \times \Z_n \times \Z_n)&\to L^2 (G \times \Z_n \times \Z_d) \\
	D_g a &= \{\sum_{m\in \Z_n} \int_{\La} a_{k,m}(\la)\pi(\la)g_{m,l} \dif \la   \}_{ k\in \Z_n, l\in\Z_d}.	
	\end{split}
	\end{equation}
	Furthermore, we define the \textit{frame-like operator} $S_{g,h} =D_h C_g $, and for brevity we write $S_g$ for $D_g C_g$. We say $S_g$ is the \textit{frame operator associated to $g$}.
	
	We say $g$ generates an \textit{$(n,d)$-matrix Gabor frame for} $L^2 (G)$ with respect to $\La$ if $S_g : L^2 (G \times \Z_n \times \Z_d) \to L^2 (G \times \Z_n \times \Z_d)$ is an isomorphism. Equivalently, the collection of time-frequency shifts
	\begin{equation}
		\GS (g; \La) := \{\pi(\la) g_{i,j} \mid \la \in \La   \}_{i\in \Z_n, j\in \Z_d}
	\end{equation}
	is a frame for $L^2 (G \times \Z_n \times \Z_d)$. We then say that $\GS (g; \La)$ is an \textit{$(n,d)$-matrix Gabor frame for $L^2 (G)$}. % This will be our preferred term.
	Equivalently, there exists $h \in L^2 (G \times \Z_n \times \Z_d)$ such that for all $f \in L^2 (G \times \Z_n \times \Z_d)$ we have
	\begin{equation}
	\label{Equation: (n,d)-matrix frames}
	f_{r,s} = \sum_{k\in \Z_d} \sum_{l \in \Z_n} \int_{\La} \langle f_{r,k} , \pi (\la)g_{l,k}\rangle \pi (\la)h_{l,s} d\la,
	\end{equation}
	for all $r \in \Z_n$ and $s \in \Z_d$. When $g$ and $h$ satisfy \eqref{Equation: (n,d)-matrix frames} we say $\GS(g;\La)$ and $\GS(h; \La)$ are a \textit{dual pair of $(n,d)$-matrix Gabor frames}. If $\La$ is implicit, we may also say $h$ is a \textit{dual $(n,d)$-matrix Gabor atom for $g$}, or just a \textit{dual atom of $g$}.
\end{defn}
\begin{rmk}
	The equivalence of the definitions of $(n,d)$-matrix Gabor frames given in \cref{Definition: nd-matrix frames} follows by \cite[Lemma 6.3.2]{Chr16} and \cref{prop: Vector in completion is Bessel} below.
\end{rmk}
\begin{rmk}
	When $\GS (g;\La)$ is an $(n,d)$-matrix Gabor frame for $L^2 (G)$, there is always a dual $(n,d)$-matrix Gabor atom for $g$, namely $h = S_g^{-1} g$. This is known as the \textit{canonical dual of $g$}.
\end{rmk}
\begin{rmk}
	One can verify that $C_g = D_g^*$. Thus $S_g$ is always a positive operator between Hilbert spaces, just as for the module frame operator in \cref{Section: Abstract Gabor Analysis}.
\end{rmk}
For general $g \in L^2 (G \times \Z_n \times \Z_d)$ the operator $C_g$ will not be bounded. Elements $g$ such that $C_g$ is bounded are of interest on their own.
\begin{defn} 
	If $g \in L^2 (G\times \Z_n \times \Z_d)$ is so that $C_g:  L^2 (G \times \Z_n \times \Z_d) \to L^2 (\La \times \Z_n \times \Z_n)$ is a bounded operator we say $g$ is an \textit{$(n,d)$-matrix Gabor Bessel vector for $L^2 (G)$ with respect to $\La$}, or that $\GS(g;\La)$ is an \textit{$(n,d)$-matrix Gabor Bessel system for $L^2 (G)$}. Equivalently, there is $D > 0$ such that for all $f \in L^2 (G \times \Z_n \times \Z_d)$ we have 
	\begin{equation}
	\label{Equation: Gabor Bessel vector inner product inequalities}
	\langle f,f \rangle \leq D \langle C_g f , C_g f \rangle,
	\end{equation}
	which may also be written as
	\begin{equation}
	\label{Equation: Gabor Bessel vector norm inequalities}
	\sum_{i \in \Z_n}\sum_{j\in \Z_d} \int_{G} \vert f_{i,j}(\xi)\vert^2 d\xi \leq D \sum_{k, l \in \Z_n} \int_{\La} \vert \sum_{m\in \Z_d}\langle f_{k,m} , \pi(\la)g_{l,m}\rangle \vert^2 d\la.
	\end{equation}
	The smallest $D>0$ such that the condition of \eqref{Equation: Gabor Bessel vector inner product inequalities} holds is called the \textit{optimal Bessel bound of $\GS(g;\La)$}, or just the \textit{optimal Bessel bound of $g$} if the set $\La$ is clear from the context.
\end{defn}

The Gabor frames of \cref{Definition: nd-matrix frames} seemingly generalize the $n$-multiwindow $d$-super Gabor frames of \cite{jalu18duality}. Indeed, we obtain $n$-multiwindow $d$-super Gabor frames if we only require reconstruction of $f \in L^2 (G \times \Z_d)$ and we identify $L^2 (G\times \Z_d)\subset L^2 (G \times \Z_n \times \Z_d)$ by embedding along a single element of $\Z_n$. Hence \eqref{Equation: (n,d)-matrix frames} generalizes both multiwindow Gabor frames and super Gabor frames as well, setting $d=1$ or $n=1$, respectively. However, we will in \cref{Proposition: nd frame is n-multiwindow d-super frame} show that any  $n$-multiwindow $d$-super Gabor frame for $L^2 (G)$ with respect to $\La$ is an $(n,d)$-matrix Gabor frame for $L^2 (G)$ with respect to $\La$. However, we continue to call them by separate names, since, as mentioned above, they are used for reconstruction in different Hilbert spaces. 

The following proposition was noted in the $(n,1)$-matrix case in \cite[Theorem 3.11]{AuEn19}, and its proof in the $(n,d)$-matrix Gabor case goes through the same except with more bookkeeping.
\begin{prop}
	\label{prop: Vector in completion is Bessel}
	Let $\La \subset \tfp{G}$ be closed and cocompact. For every $g \in M_{n,d}(E)$, $C_g :L^2 (G \times \Z_n \times \Z_d) \to L^2 (\La \times \Z_n \times \Z_n)$ is a bounded operator. In other words, every $g \in M_{n,d}(E)$ is a Bessel vector.
\end{prop}
For ease of notation, the localization map in in $M_n (A)$ will be denoted by $\rho_{M_n (A)}$, even though we might not be able to localize all of $M_n (A)$. With the above definitions, the following calculation is justified for $f,g \in M_{n,d}(E) \subset L^2 (G \times \Z_n \times \Z_d)$ by \cref{prop: Vector in completion is Bessel}
\begin{equation*}
\begin{split}
\rho_{M_{n}(A)} \modan_g (f) &= \rho_{M_{n}(A)} (\bbracket f,g ]) \\
&= \rho_{M_{n}(A)} (\{ \sum_{m\in \Z_d} \int_{\La} \langle f_{k,m}, \pi (\la) g_{l,m}\rangle \pi (\la) \}_{k,l \in \Z_n}) \\
&= \{ \sum_{m\in \Z_d}  \langle f_{k,m}, \pi (\la) g_{l,m}\rangle \}_{\la \in \La ,k,l \in \Z_n} = C_g \rho_{M_{n,d}(E)}(f).
\end{split}
\end{equation*}
Hence we obtain the following result. 

\begin{lemma}
	\label{Lemma: Coefficient operator commutation}
	Let $\La \subset \tfp{G}$ be closed and cocompact. For every $g \in M_{n,d} (E)$, the module coefficient operator $\modan_g$ localizes to give the coefficient operator $C_g$. Equivalently, the diagram
	\begin{center}
		\begin{tikzpicture}[scale=1.0]
		\node (1) at (0,2) {$M_{n,d}(E)$};
		\node (2) at (6,2) {$M_n (A)$};
		\node (3) at (0,0) {$L^2 (G \times \Z_n \times \Z_d)$};
		\node (4) at (6,0) {$L^2 (\La \times \Z_n \times \Z_n)$};
		\path[->,font=\scriptsize]
		(1) edge[auto] node[auto] {$\modan_g$} (2)
		(1) edge node[auto,swap] {$\rho_{M_{n,d}(E)}$} (3)
		(2) edge node[auto] {$\rho_{M_n (A)}$} (4)
		(3) edge node[auto] {$C_g$} (4);
		\end{tikzpicture}
	\end{center}
	commutes for all $g \in M_{n,d}(E)$.
\end{lemma} 
Likewise one may obtain $C_g^* \rho_{M_{n}(A)} = \rho_{M_{n,d} (E)}\modan_g^* : M_{n} (\blangle E,E \rangle) \to L^2 (G \times \Z_n \times \Z_d)$ for all $g \in M_{n,d} (E)$. Note that the domain might be larger, but we cannot guarantee this unless $A$ is unital, that is, when $\La$ is discrete.
\begin{lemma} 
	\label{Lemma: Synthesis operator commutation}
	Let $\La \subset \tfp{G}$ be closed and cocompact. For every $g\in M_{n,d}(E)$, the module synthesis operator $\modan_g^*$ localizes to the Gabor synthesis operator $C_g^*$. Equivalently, the diagram
	\begin{center}
		\begin{tikzpicture}[scale=1.0]
		\node (1) at (0,2) {$M_{n}(\blangle E,E \rangle)$};
		\node (2) at (6,2) {$M_{n,d} (E)$};
		\node (3) at (0,0) {$L^2 (\La \times \Z_n \times \Z_n)$};
		\node (4) at (6,0) {$L^2 (G \times \Z_n \times \Z_d)$};
		\path[->,font=\scriptsize]
		(1) edge[auto] node[auto] {$\modan_g^*$} (2)
		(1) edge node[auto,swap] {$\rho_{M_n (A)}$} (3)
		(2) edge node[auto] {$\rho_{M_{n,d}(E)}$} (4)
		(3) edge node[auto] {$C_g^*$} (4);
		\end{tikzpicture}
	\end{center}
	commutes for every $g \in M_{n,d}(E)$.
\end{lemma}
Combining \cref{Lemma: Coefficient operator commutation} and \cref{Lemma: Synthesis operator commutation} we then obtain
\begin{prop}
	\label{Proposition: Module frame localizes to Gabor frame}
	Let $\La \subset \tfp{G}$ be closed and cocompact. For all $g,h \in M_{n,d}(E)$, $S_{g,h} \rho_{M_{n,d}(E)} = \rho_{M_{n,d}(E)} \modft_{g,h}$, meaning the module frame-like operator $\modft_{g,h}$ localizes to the frame-like operator $S_{g,h}$. Equivalently, the diagram
	\begin{center}
		\begin{tikzpicture}[scale=1.0]
		\node (1) at (0,2) {$M_{n,d}(E)$};
		\node (2) at (6,2) {$M_{n,d} (E)$};
		\node (3) at (0,0) {$L^2 (G \times \Z_n \times \Z_d)$};
		\node (4) at (6,0) {$L^2 (G \times \Z_n \times \Z_d)$};
		\path[->,font=\scriptsize]
		(1) edge[auto] node[auto] {$\modft_{g,h}$} (2)
		(1) edge node[auto,swap] {$\rho_{M_{n,d}(E)}$} (3)
		(2) edge node[auto] {$\rho_{M_{n,d}(E)}$} (4)
		(3) edge node[auto] {$S_{g,h}$} (4);
		\end{tikzpicture}
	\end{center}
	commutes for all $g \in M_{n,d} (E)$. 
\end{prop}
As $\rho_{M_{n,d} (E)} : M_{n,d}(E) \to \rho_{M_{n,d} (E)}(M_{n,d}(E))$ is a linear bijection intertwining both the $A$-actions and the $B$-actions, we see by \cref{Proposition: Module frame localizes to Gabor frame} that for $g\in M_{n,d} (E)$, $\modft_g$ is invertible if and only if $S_g \vert_{\rho_{M_{n,d} (E)}(M_{n,d}(E))}$ is invertible. But we also have the following result.
\begin{lemma}
	\label{Lemma: Localization frame operator extends}
	Let $\La \subset \tfp{G}$ be closed and cocompact, and let $g \in M_{n,d} (E)$. Then
	\begin{equation*}
	S_g \vert_{\rho_{M_{n,d} (E)}(M_{n,d}(E))}: \rho_{M_{n,d} (E)}(M_{n,d}(E)) \to \rho_{M_{n,d} (E)}(M_{n,d}(E))
	\end{equation*}
	is invertible if and only if 
	\begin{equation*}
	S_g : L^2 (G \times \Z_n \times \Z_d) \to L^2 (G \times \Z_n \times \Z_d)
	\end{equation*}
	is invertible. %with $S_g^{-1} g \in \rho_{M_{n,d}(E)}(M_{n,d} (E))$.
\end{lemma}
\begin{proof}
	Suppose first $S_g \vert_{\rho_{M_{n,d} (E)}(M_{n,d}(E))}: \rho_{M_{n,d} (E)}(M_{n,d}(E)) \to \rho_{M_{n,d} (E)}(M_{n,d}(E))$ is invertible. Since any $g \in M_{n,d}(E)$ is a Bessel vector by \cref{prop: Vector in completion is Bessel}, we may extend the operator by continuity to obtain that $S_g : L^2 (G \times \Z_n \times \Z_d) \to L^2 (G \times \Z_n \times \Z_d)$ is invertible as well.
	
	Conversely, suppose $S_g : L^2 (G \times \Z_n \times \Z_d) \to L^2 (G \times \Z_n \times \Z_d)$ is invertible. Since $\ft_g$ is the continuous extension of $\modft_g$, it then follows by \cref{Localization norm preserved} and \cref{Proposition: C*-subalgebra inverse closed} that $\Theta_g$ is invertible, which implies $S_g \vert_{\rho_{M_{n,d} (E)}(M_{n,d}(E))}$ is invertible.
\end{proof}
\begin{rmk}
	From now on we will identify $M_{n,d}(E)$ and its image in the localization, and we will do this without mention. %The same goes for $M_n (\blangle E,E\rangle)$ (and potentially a larger domain), and $M_d (B)$.
\end{rmk}
%\begin{rmk}
%	The assumption $S_g^{-1}g\in M_{n,d}(E)$ in \cref{Lemma: Localization frame operator extends} may seem quite restrictive. However, we will see that whenever $\La$ is cocompact, that is, whenever $B$ is unital, this is automatically satisfied, see \cref{Proposition: Dual atom of regular is regular when discrete}.
%\end{rmk}
Combining \cref{Proposition: Module frame localizes to Gabor frame} and \cref{Lemma: Localization frame operator extends} we obtain the following important result.% which applies under the standing assumption that $\La$ is cocompact.
\begin{prop}
	\label{Proposition: Module frame iff Gabor frame}
	Let $\La \subset \tfp{G}$ be closed and cocompact. For $g \in M_{n,d}(E)$ we have that $\modft_g$ is invertible if and only if $S_g$ is invertible. % with canonical dual atom in $M_{n,d}(E)$. 
	In other words, $g$ generates a module $(n,d)$-matrix frame for $E$ with respect to $A$ if and only $\GS(g;\La)$ is an $(n,d)$-matrix Gabor frame for $L^2 (G)$.% with canonical dual window in $M_{n,d}(E)$.
\end{prop}
We also have the following important corollary.
\begin{corl}
	\label{Corollary: Dual Gabor frames iff extend to identity}
	Let $\La \subset \tfp{G}$ be closed and cocompact, and let $g,h \in M_{n,d} (E)$. Then $g$ and $h$ generate dual (n,d)-matrix Gabor frames for $L^2 (G)$ with respect to $\La$ if and only if $[g,h\bracketb$ extends to the identity operator on $L^2 (G\times \Z_n \times \Z_d)$.	
\end{corl}
\begin{proof}
	Suppose first $g,h \in M_{n,d}(E)$ generate dual $(n,d)$-matrix Gabor frames for $L^2 (G)$ with respect to $\La$. Then we know that for all $f \in M_{n,d}(E)$ we have
	\begin{equation*}
	f = \bbracket f,g ]h = f [g,h\bracketb,
	\end{equation*}
	from which we as before deduce that $[g,h\bracketb = 1_{M_d (B)}$. This extends by continuity to the identity operator on all of $L^2 (G \times \Z_n \times \Z_d)$.
	
	Conversely, if $[g,h\bracketb$ extends to the identity operator on $L^2 (G \times \Z_n \times \Z_d)$, then $[g,h\bracketb$ acts as the identity on $M_{n,d}(E)$. For any $f \in M_{n,d}(E)$ we then have 
	\begin{equation*}
	f = f [g,h\bracketb = \bbracket f,g]h,
	\end{equation*}
	hence \eqref{Equation: (n,d)-matrix frames} holds for all $f \in M_{n,d}(E)$. But this extends to $L^2 (G \times \Z_n \times \Z_d)$ by continuity, which implies that $g$ and $h$ generate dual $(n,d)$-matrix Gabor frames.
\end{proof}
We wish to establish a duality principle for $(n,d)$-matrix Gabor frames. For this we also need to treat $(n,d)$-matrix Gabor Riesz sequences and relate them to \cref{Definition: Modular Riesz sequence}. 
\begin{defn}
	\label{Definition: nd-matrix Gabor Riesz sequence}
	Let $g \in L^2 (G \times \Z_n \times \Z_d)$. We say $g$ generates an \textit{$(n,d)$-matrix Gabor Riesz sequence for $L^2 (G)$ with respect to $\La$}, or that $\GS(g;\La)$ is an \textit{$(n,d)$-matrix Gabor Riesz sequence for $L^2 (G)$,} if 
	\begin{equation}
	C_g C_g^*: L^2 (\La \times \Z_n \times \Z_n) \to L^2 (\La \times \Z_n \times \Z_n)
	\end{equation}
	is an isomorphism. Equivalently, there exists $h\in L^2 (G \times \Z_n \times \Z_d)$ such that for all $a \in L^2 (\La \times \Z_n \times \Z_n)$ we have
	\begin{equation}
	\label{Equation: (n,d)-matrix Gabor Riesz sequence}
	a_{r,s}(\mu) = \sum_{i \in \Z_d} \sum_{j\in \Z_n} \langle \int_{\La} a_{r,j}(\la) \pi(\la) g_{j,i} d\la, \pi (\mu) h_{s,i}\rangle
	\end{equation}
	for all $r,s \in \Z_n$ and all $\mu \in \La$. If \eqref{Equation: (n,d)-matrix Gabor Riesz sequence} is satisfied we will say $h$ generates a \textit{dual $(n,d)$-matrix Gabor Riesz sequence of $g$}.  
\end{defn}
\begin{rmk}
	Note that the equivalence of the definitions of $(n,d)$-matrix Gabor Riesz sequences in \cref{Definition: nd-matrix Gabor Riesz sequence} follows by \cite[Theorem 3.6.6]{Chr16} and \cref{prop: Vector in completion is Bessel}.
\end{rmk}
\begin{rmk}
	\eqref{Equation: (n,d)-matrix Gabor Riesz sequence} can be seen to be equivalent to $C_g C_h^* = C_h C_g^* = \Id_{L^2 (\La \times \Z_n \times \Z_n)}$.
\end{rmk}

Before treating localization of module matrix Riesz sequences and how they relate to matrix Gabor Riesz sequences, we do a necessary but justified simplification. Recall that existence of finite module matrix Riesz sequences for $M_{n,d}(E)$ with respect to $M_n (A)$ requires $A$ to be  unital by \cref{Proposition: Modular Ron-Shen Duality}. In the following we therefore let $\La$ be discrete, but not necessarily cocompact. Hence in the following, $A$ is unital with a faithful trace, but $B$ might not have that property. By \cite[p. 251]{JaLe16} we know that $\GS (g; \La)$ is a Bessel system with Bessel bound $D$ if and only if $\GS (g; \Lao)$ is a Bessel system with Bessel bound $D$. Applying \cref{prop: Vector in completion is Bessel} we immediately get the following from \cref{Lemma: Coefficient operator commutation} and \cref{Lemma: Synthesis operator commutation}.
\begin{prop}
	\label{Proposition: Localization of modular Riesz sequence is Riesz sequence for image}
	Let $\La \subset \tfp{G}$ be discrete. For all $g,h \in M_{n,d} (E)$ we have $(C_h C_g^*) \circ \rho_{M_n (A)} = \rho_{M_n (A)}\circ (\modan_h \modan_g^*)$. Equivalently, the diagram
	\begin{center}
		\begin{tikzpicture}[scale=1.0]
		\node (1) at (0,2) {$M_{n}(A)$};
		\node (2) at (6,2) {$M_{n} (A)$};
		\node (3) at (0,0) {$L^2 (\La \times \Z_n \times \Z_n)$};
		\node (4) at (6,0) {$L^2 (\La \times \Z_n \times \Z_n)$};
		\path[->,font=\scriptsize]
		(1) edge[auto] node[auto] {$\modan_h \modan_g^*$} (2)
		(1) edge node[auto,swap] {$\rho_{M_{n}(A)}$} (3)
		(2) edge node[auto] {$\rho_{M_{n}(A)}$} (4)
		(3) edge node[auto] {$C_h C_g^*$} (4);
		\end{tikzpicture}
	\end{center}
	commutes. 
\end{prop}
As $\rho_{M_{n} (A)} : M_{n}(A) \to \rho_{M_{n} (A)}(M_{n}(A))$ is a linear bijection respecting the actions of $A$, we see by \cref{Proposition: Localization of modular Riesz sequence is Riesz sequence for image} that for $g\in M_{n,d} (E)$, $\modan_g \modan_g^*$ is an isomorphism if and only if $(C_g C_g^*)\vert_{\rho_{M_{n} (A)}(M_{n}(A))}$ is an isomorphism. In analogy with \cref{Lemma: Localization frame operator extends} we have the following result.
\begin{lemma}
	\label{Lemma: Localization of Riesz sequence extends to Riesz sequence}
	Let $\La \subset \tfp{G}$ be discrete. For $g \in M_{n,d}(E)$ we have that 
	\begin{equation}
	(C_gC_g^*)\vert_{\rho_{M_n (A)}(M_n (A))}: \rho_{M_n (A)}(M_n (A)) \to \rho_{M_n (A)}(M_n (A))
	\end{equation}
	is invertible if and only if
	\begin{equation}
	C_g C_g^* \colon L^2 (\La \times \Z_n \times \Z_n) \to L^2 (\La\times \Z_n \times \Z_n)
	\end{equation}
	is invertible.
\end{lemma}
\begin{proof}
	Suppose first that $(C_gC_g^*)\vert_{\rho_{M_n (A)}(M_n (A))}: \rho_{M_n (A)}(M_n (A)) \to \rho_{M_n (A)}(M_n (A))$ is invertible. Since any $g \in M_{n,d}(E)$ is a Bessel vector by \cref{prop: Vector in completion is Bessel}, we may extend the operator by continuity to obtain that $C_g C_g^* \colon L^2 (\La \times \Z_n \times \Z_n) \to L^2 (\La\times \Z_n \times \Z_n)$ is invertible as well.
	
	Conversely, suppose $C_g C_g^* \colon L^2 (\La \times \Z_n \times \Z_n) \to L^2 (\La\times \Z_n \times \Z_n)$ is invertible. Since $C_g C_g^*$ is the continuous extension of $\modan_g \modan_g^*$, it then follows by  \cref{Localization norm preserved} and \cref{Proposition: C*-subalgebra inverse closed} that $\modan_g \modan_g^*$ is invertible as well, which implies $(C_gC_g^*)\vert_{\rho_{M_n (A)}(M_n (A))}: \rho_{M_n (A)}(M_n (A)) \to \rho_{M_n (A)}(M_n (A))$ is invertible.
\end{proof}
\begin{rmk}
	From now on we will identify $M_n (\blangle E,E\rangle)$ (and potentially a larger domain) and its localization. The same goes for $M_d (B)$.
\end{rmk}
Now the following is an immediate consequence.
\begin{prop} 
	\label{Proposition: Module Riesz sequence iff Gabor Riesz sequence}
	Let $\La \subset \tfp{G}$ be discrete. For $g \in M_{n,d}(E)$ we have that $\modan_g\modan_g^*$ is invertible if and only if $C_g C_g^* $ is invertible. In other words, $g$ generates a module $(n,d)$-matrix Riesz sequence for $E$ with respect to $A$ if and only if $\GS(g; \La)$ is an $(n,d)$-matrix Gabor Riesz sequence for $L^2 (G)$.
	%there is $h \in M_{n,d}(E)$ such that $C_h C_g^* = \Id_{L^2 (\La \times \Z_n \times \Z_n)}$. In other words, $g$ is a module $(n,d)$-matrix Riesz sequence for $E$ with respect to $A$ if and only if $\GS(g;\La)$ is an $(n,d)$-matrix Gabor Riesz sequence for $L^2 (G)$ with a dual $(n,d)$-matrix Gabor Riesz sequence $\GS (h;\La)$ with the property $h \in M_{n,d}(E)$.
\end{prop}
By the proof of \cref{Lemma: Localization of Riesz sequence extends to Riesz sequence} we then have the following statement.
\begin{corl}
	\label{Dual Gabor Riesz sequences iff extend to identity}
	Let $\La \subset \tfp{G}$ be discrete. Suppose $g,h \in M_{n,d} (E)$. Then $g$ and $h$ generate dual $(n,d)$-matrix Gabor Riesz sequences for $L^2 (G)$ with respect to $\La$ if and only if $\bbracket g,h ]$ extends to the identity operator on $L^2 (\La \times \Z_n \times \Z_n)$. 
\end{corl}
\begin{proof}
	Suppose first that $g$ and $h$ generate dual $(n,d)$-matrix Gabor Riesz sequences for $L^2 (G)$ with respect to $\La$. Then for all $a \in M_n (A)$ we have
	\begin{equation*}
	(a_{r,s}) = \{\sum_{i \in \Z_d} \sum_{j\in \Z_n} \langle \int_{\La} a_{r,j}(\la) \pi(\la) g_{j,i} d\la, \pi (\mu) h_{s,i}\rangle\}_{\mu \in \La, r,s\in \Z_n},
	\end{equation*}
	which is equivalent to $a= a\bbracket g,h] $ for all $a \in M_n (A)$. But the %latter 
	first expression extends by continuity to $L^2 (\La \times \Z_n \times \Z_n)$, so $\bbracket g,h]$ extends to the identity on $L^2 (\La \times \Z_n \times \Z_n)$.
	
	Conversely, suppose $\bbracket g,h ]$ extends to the identity on $L^2 (\La \times \Z_n \times \Z_n)$. Once again, for all $a \in M_n (A)$ we then have
	\begin{equation*}
	(a_{r,s}) = \{\sum_{i \in \Z_d} \sum_{j\in \Z_n} \langle \int_{\La} a_{r,j}(\la) \pi(\la) g_{j,i} d\la, \pi (\mu) h_{s,i}\rangle\}_{\mu \in \La, r,s\in \Z_n},
	\end{equation*}
	which again extends to $L^2 (G \times \Z_n \times \Z_n)$. Hence $g$ and $h$ are dual $(n,d)$-matrix Gabor Riesz sequences for $L^2 (G)$ with respect to $\La$.
\end{proof}
	Note how the above results guarantee that when $\La \subset \tfp{G}$ is closed and cocompact and $g \in M_{n,d}(E)$ is such that $\GS(g;\La)$ is an $(n,d)$-matrix Gabor frame for $L^2 (G)$, the canonical dual frame $S_g^{-1}g \in M_{n,d}(E)$. Indeed,
	\begin{equation*}
		S_g^{-1} g = \modft_g^{-1} g = [g,g\bracketb^{-1} \in M_{n,d} (E).
	\end{equation*}
	Likewise, for Riesz sequences there is the notion of \textit{canonical biorthogonal atom}, see for example \cite[p. 160]{Chr16}. Restricting to $\La$ discrete, it is given by $(S^{B}_g)^{-1}g$, where $S^{B}_g$ is the frame operator with respect to the right hand side, that is, with respect to $\Lao$. We see that for all $f \in M_{n,d} (E)$
	\begin{equation*}
		S^{B}_g f = (\modan^B_g)^* \modan^B_g f =  (\modan^B_g)^* ([g,f\bracketb) = g[g,f\bracketb = \bbracket g,g] f.
	\end{equation*}
	Thus it follows that
	\begin{equation*}
		(S^{B}_g)^{-1}g = (\modft^B_g)^{-1}g  = \bbracket g,g ]^{-1} g \in M_{n,d}(E).
	\end{equation*}
	Hence for both matrix Gabor frames and matrix Gabor Riesz sequences with generating atom in $M_{n,d}(E)$, the canonically associated dual atoms are also in $M_{n,d} (E)$.
	%For both $(n,d)$-matrix Gabor frames and $(n,d)$-matrix Gabor Riesz sequences with $g \in M_{n,d} (E)$, it is true that $g$ has the canonical dual atom in $M_{n,d} (E)$ if $g$ has a dual atom in $M_{n,d} (E)$. Let us make this clear for $(n,d)$-matrix Gabor frames. An analogous argument works for $(n,d)$-matrix Gabor Riesz sequences. If $g$ has a dual atom $h$ in $M_{n,d} (E)$, we may deduce that $M_{n,d} (E)$ is singly generated by $g$ as an $M_n (A)$-module, which implies that $M_d (B)$ is unital and $[g,g\bracketb$ is invertible with inverse $[h,h\bracketb$. This allows us to make the canonical dual atom $g [g,g\bracketb^{-1}$ as before, and it will clearly be in $M_{n,d} (E)$.

%\begin{rmk}
	%The above statements can of course also be made with $B$ (that is, with respect to $\Lao$)  instead of $A$ (that is, with respect to $\La$).
%\end{rmk}
%
We have the following result which shows that in the cases we are interested in, if the generating atom is regular, the canonical dual atom has the same regularity.
\begin{prop}
	\label{Proposition: Dual atom of regular is regular when discrete}
	Let $g \in M_{n,d}(\E)$.
	\begin{enumerate}
		\item [i)] If $\GS (g; \La)$ is an $(n,d)$-matrix Gabor frame for $L^2 (G)$ and $\La$ is closed and cocompact in $G\times \widehat{G}$, then the canonical dual atom is in $M_{n,d}(\E)$. %$h = S_g^{-1} g \in M_{n,d}(\E)$. 
		\item [ii)] If $\GS (g; \La)$ is an $(n,d)$-matrix Gabor Riesz sequence for $L^2 (G)$ and $\La$ is discrete, then the canonical biorthogonal atom is also in $M_{n,d} (\E)$.
	\end{enumerate}
\end{prop}
\begin{proof}
	For the proof of i), note that the assumption that $\La$ is cocompact implies that $\Lao$ is discrete, so by \cref{Lemma: Group C*-alg unital iff discrete} we get that $M_d (B)$ is unital. Also $M_d (B)$ is a $C^*$-subalgebra of $\mathbb{B}(H_{M_{n,d}(E)})$ by \cref{Localization norm preserved}. 
	That $\GS (g; \La)$ is an $(n,d)$-matrix Gabor frame for $L^2 (G)$ then means that \eqref{Equation: Localization of modular frame} is satisfied for our current setting. We deduce, as in the proof of \cref{Proposition: Dual atom in E iff in H_E}, that $[g,g\bracketb$ is invertible in $M_d (B)$. Since $[g,g\bracketb \in M_d (\B)$ and $M_d (\B)$ is spectral invariant in $M_d (B)$ by \cref{Proposition: Spectral invariance when discrete} and \cref{Lemma: Spectral invariance lifts to matrices}, it follows that
	the canonical dual atom is $g [g,g\bracketb^{-1} \in M_{n,d}(\E)$. %The operator $[g,h\bracketb$ extends to the identity on all of $L^2 (G \times \Z_n \times \Z_d)$.
	%If $g \in \E$, then $[g,g\bracketb \in M_d (\B)$, and as $\Lao$ is discrete, $\B$ is spectral invariant in $B$ by \cref{Proposition: Spectral invariance when discrete}. We then know that $M_d (\B)$ is spectral invariant in $M_d (B)$, so $[g,g\bracketb^{-1} \in M_d (\B)$. Hence $h = g [g,g\bracketb^{-1} \in M_{n,d}(\E)$.
	
	For the proof of ii), note that the assumption that $\La$ is discrete implies $M_n (A)$ is unital. 
	Also, $M_n (A)$ is a $C^*$-subalgebra of $\mathbb{B}(H_{M_n (A)})$ by \cref{Localization norm preserved}.
	That $\GS (g; \La)$ determines an $(n,d)$-matrix Gabor Riesz sequence for $L^2 (G)$ then means that \eqref{Equation: Riesz trace inequality} is satisfied for our current setting. The middle term of \eqref{Equation: Riesz trace inequality} can be written $(a \bbracket g,g] , a)_A$, so $\bbracket g,g]$ extends to a positive, invertible operator on $H_{M_n (A)}$. We deduce as in the proof of \cref{Proposition: Modular Riesz sequences} that $\bbracket g,g ]$ is invertible in $M_n (A)$. Since $g \in M_{n,d}(\E)$ we have $\bbracket g,g] \in M_{n}(\A)$, and by \cref{Proposition: Spectral invariance when discrete} and \cref{Lemma: Spectral invariance lifts to matrices} $M_n (\A)$ is spectral invariant in $M_n (A)$. It follows that 
	the canonical dual atom $h:=\bbracket g,g ]^{-1} g$ is in $M_{n,d}(\E)$. %and the operator $\bbracket g,h]$ extend to the identity operator on all of $L^2 (\La \times \Z_n \times \Z_n)$.
	%
	%If $g \in \E$, then $\bbracket g,g] \in M_n (\A)$, and $M_n (\A)$ is spectral invariant in $M_n (A)$ since $\La$ is discrete. Hence $h = \bbracket g,g]^{-1} g \in M_{n,d}(\E)$.
\end{proof}
\begin{rmk}
	In the special case $n=d=1$ \cref{Proposition: Dual atom of regular is regular when discrete} gives a proof of the fact that the canonical dual atom of a Gabor frame vector in Feichtinger's algebra  $S_0 (G)$ is also in Feichtinger's algebra whenever $\La$ is cocompact. 
\end{rmk}

When applying the module setup of \cref{Section: Abstract Gabor Analysis} to Gabor analysis, we take as a pre-equivalence bimodule $\E= S_0(G \times \Z_n \times \Z_d)$, which is a proper subspace of $L^2 (G \times \Z_n \times \Z_d)$ unless $G$ is a finite group. Even the Hilbert $C^*$-module completion $E$ is properly contained in $L^2 (G \times \Z_n \times \Z_d)$ for most choices of $\La$, see \cite[Example 3.8]{AuEn19}. As such, we cannot hope to treat general atoms in $L^2 (G \times \Z_n \times \Z_d)$ by applying just this method. But indeed the module reformulation is made exactly to guarantee some regularity of the atoms generating frames. %Indeed, we know from \cref{prop: Vector in completion is Bessel} that any $g \in M_{n,d} (E)$ is a Bessel vector in the localization. With the current approach we can therefore only prove results on $(n,d)$-matrix Gabor frames under the assumption that $g$ and a dual atom are both in $M_{n,d}(E)$. The reconstruction formula of \eqref{Equation: (n,d)-matrix frames} above does however pass to the localization, meaning that the module $(n,d)$-matrix frames we treat will be $(n,d)$-matrix Gabor frames for $L^2 (G)$. The analogous statement is of course true for $(n,d)$-matrix Gabor Riesz sequences. In light of this, we proceed to show that the only interesting setting for us is when $\La \subset G \times \widehat{G}$ is cocompact. Furthermore, we show that whenever $\La$ is cocompact and $g \in M_{n,d}(E)$ is such that $\GS (g;\La)$ is an $(n,d)$-matrix Gabor frame for $L^2(G)$, the canonical dual is in $M_{n,d}(E)$, as well as the corresponding statement for $(n,d)$-matrix Gabor Riesz sequences.

From \cref{Definition: nd-matrix frames} we see that $(n,d)$-matrix Gabor frames generalize $n$-multiwindow $d$-super Gabor frames considered in \cite{jalu18duality}. However, we now make clear how they fit into the module framework. As mentioned earlier, we obtain $n$-multiwindow $d$-super Gabor frames if we only require reconstruction of $f \in L^2 (G \times \Z_d)$ and we identify $L^2 (G\times \Z_d)\subset L^2 (G \times \Z_n \times \Z_d)$ by embedding it along a single element in $\Z_n$. The module reformulation of this is that $g, h \in M_{n,d}(E)$ are dual $n$-multiwindow $d$-super Gabor frames if for all $f \in M_{n,d}(E)$ supported only one row we have
\begin{equation}
\label{Equation: Module reformulation of n-multiwindow d-super frame}
f = \bbracket f,g ] h = f [g,h\bracketb.
\end{equation}
%By how we defined the module actions, this happens if and only if $[g,h\bracketb = 1_{M_d (B)}$. The operator $[g,h\bracketb$ of course extend to the identity on all of $L^2 (G \times \Z_n)$.

Likewise, it is clear that the $(n,d)$-matrix Gabor Riesz sequences of \cref{Definition: nd-matrix Gabor Riesz sequence} generalize the $n$-multiwindow $d$-super Gabor Riesz sequences also considered in \cite{jalu18duality}. Indeed, we obtain $n$-multiwindow $d$-super Gabor Riesz sequences if we only require reconstruction of $a \in L^2 (\La \times \Z_n)$ and we identify $L^2 (\La \times \Z_n) \subset L^2 (\La \times \Z_n \times \Z_n)$ by embedding it along a single element in the middle copy of $\Z_n$. The module reformulation of this is that $g,h \in M_{n,d}(E)$ are dual $n$-multiwindow $d$-super Gabor Riesz sequences if for all $a \in M_n (A)$ supported only one row we have
\begin{equation}
\label{Equation: Module reformulation of n-multiwindow d-super Riesz}
a = \bbracket ag,h] = a \bbracket g,h].
\end{equation}
We proceed to prove that all $n$-multiwindow $d$-super Gabor frames for $L^2 (G)$ with respect to $\La$ are $(n,d)$-matrix Gabor frames for $L^2 (G)$ with respect to $\La$, as well as the analogous statement for Riesz sequences. The converse statement is true as well.
\begin{prop}
	\label{Proposition: nd frame is n-multiwindow d-super frame}
	Let $g$ be in $M_{n,d} (E)$. 
	\begin{enumerate}
		\item [i)] If $\GS(g;\La)$ is an $n$-multiwindow $d$-super Gabor frame for $L^2 (G)$ with a dual window $h \in M_{n,d}(E)$, then $\GS(g;\La)$ is an $(n,d)$-matrix Gabor frame for $L^2 (G)$ with dual window $h$.
		\item [ii)] If $\GS(g;\La)$ is an $n$-multiwindow $d$-super Gabor Riesz sequence for $L^2 (G)$ with a dual Gabor Riesz sequence $\GS(h;\La)$ with $h \in M_{n,d}(E)$, then $\GS(g;\La)$ is an $(n,d)$-matrix Gabor Riesz sequence for $L^2 (G)$ with dual Gabor Riesz sequence $\GS (h;\La)$.% in $M_{n,d} (E)$.
	\end{enumerate}
\end{prop}
\begin{proof}
	If $\GS(g; \La)$ is an $n$-multiwindow $d$-super Gabor frame for $L^2 (G)$ with respect to $\La$ with a dual window $h \in M_{n,d}(E)$, we can, as noted above, reconstruct any $f \in M_{n,d}(E)$ supported on a single row. In other words,
	\begin{equation*}
	f = f [g,h\bracketb
	\end{equation*}
	for all $f \in M_{n,d}(E)$ supported on a single row. Given arbitrary $f' \in M_{n,d}(E)$ we may then just write $f'$ as a sum of $n$ matrices $f'_i$, $i=0 , \ldots ,n-1,$ with only one nonzero row, namely the kth row of $f'_i$ is given by
	\begin{equation*}
	 (f_{i,0}, \ldots , f_{i,d-1})\delta_{ik}
	\end{equation*}
	for $k \in \Z_n$,
	and we can reconstruct each of these rows. Hence we can reconstruct arbitrary elements of $M_{n,d}(E)$. This passes to the localization $L^2 (G \times \Z_n \times \Z_d)$, and thus finishes the proof of (i).
	
	The proof of (ii) is completely analogous, writing elements $a \in M_n (A)$ as a sum of matrices with only one nonzero row and then using that we can reconstruct such matrices. This will also pass to the localization $L^2 (\La \times \Z_n \times \Z_n)$.
\end{proof}
%
%By Proposition \cref{Proposition: If dual is in module, La is cocpt}, Proposition \cref{Proposition: Dual atom of regular is regular when discrete}, and Proposition \cref{Proposition: (n,d) frame is n-multiwindow d-super frame} we obtain the following two corollaries.
%\begin{corl}
%	\label{Corollary: Existence n-mw d-super frame iff cocompact}
%	Let $\La$ be a closed subgroup of $G \times \widehat{G}$. There exist $n,d \in \N$ and $g \in M_{n,d}(E)$ such that $g$ is an $n$-multiwindow $d$-super Gabor frame for $L^2 (G)$ with respect to $\La$ with canonical dual window in $M_{n,d}(E)$ if and only if $\La$ is cocompact in $G \times \widehat{G}$. 
%\end{corl}
%
%\begin{corl}
%	\label{Corollary: Existence n-mw d-super Riesz iff discrete}
%	Let $\La$ be a closed subgroup of $G \times \widehat{G}$. There exist $n,d \in \N$ and $g \in M_{n,d}(E)$ such that $g$ is an $n$-multiwindow $d$-super Gabor Riesz sequence for $L^2 (G)$ with respect to $\La$ with a dual window in $M_{n,d}(E)$ if and only if $\La$ is discrete.  
%\end{corl}

Given a closed and cocompact subgroup $\La$, we may ask if there are restrictions on $n,d\in \N$ for there to possibly exist $(n,d)$-matrix Gabor frames for $L^2 (G)$ with respect to $\La$. Conversely, if we fix $n$ and $d$, we may ask if there are restrictions on the size of the subgroup $\La$ for there to possibly exist $(n,d)$-matrix Gabor frames for $L^2 (G)$ with respect to $\La$. When $\La$ is a lattice, we have the following proposition.
%If we fix a lattice $\La$, that is $\La$ is both cocompact and discrete, we may ask if there are restrictions on $n,d\in \N$ for there to possibly exist $(n,d)$-matrix Gabor frames for $L^2 (G)$ with respect to $\La$. Conversely, if we fix $n$ and $d$, we may ask if there are restrictions on the size of the lattice $\La$ for there to possibly exist $(n,d)$-matrix Gabor frames for $L^2 (G)$ with respect to $\La$. This is the content of the following proposition.
\begin{prop}
	\label{Proposition: Existence of frames vs lattice size}
	Let $\La \subset G \times \widehat{G}$ be a lattice. If there is $g \in M_{n,d}(E)$ such that $\GS(g;\La)$ is an $(n,d)$-matrix Gabor frame for $L^2 (G)$, then 
	\begin{equation*}
	s(\La) \leq \frac{n}{d}.
	\end{equation*}
\end{prop}
\begin{proof}
	Since $\La$ is discrete and cocompact, both $A$ and $B$ are unital. We also know by \cref{Proposition: Dual atom of regular is regular when discrete} that the canonical dual of $g$ is in $M_{n,d}(E)$. Hence we are in the setting of \cref{thm:req-on-n-d-for-mod-frame}. Since module $(n,d)$-matrix frames localize to $(n,d)$-matrix Gabor frames for the localization, and we have $\tr_{A} (1_{A} ) =1$, and $\tr_{B} (1_{B}) = s(\La)$ (since the identity on $B$ is $s(\La)\delta_0$, where $\delta_0$ is the indicator function in the group identity, see for example \cite{ri88}), the result is immediate by \cref{thm:req-on-n-d-for-mod-frame}.
\end{proof}
Likewise, given a lattice $\La$, we may ask if there is a relationship between the size of $\La$ and the integers $n$ and $d$ such that there can possibly exist $(n,d)$-matrix Gabor Riesz sequences for $L^2 (G)$ with respect to $\La$. This is the content of the following proposition.
\begin{prop}
	\label{Proposition: Existence of Riesz vs lattice size}
	Let $\La \subset G \times \widehat{G}$ be a lattice. If $g \in M_{n,d} (E)$ is such that $\GS(g;\La)$ is an $(n,d)$-matrix Gabor Riesz sequence for $L^2 (G)$, then
	\begin{equation*}
	s (\La) \geq \frac{n}{d}.
	\end{equation*}
\end{prop}
\begin{proof}
	As before we know by the conditions on $\La$ that both $A$ and $B$ are unital, and by \cref{Proposition: Dual atom of regular is regular when discrete} the canonical dual of $g$ is in $M_{n,d}(E)$. Thus we are in the setting of \cref{Theorem: Requirements for module Riesz}. Since module $(n,d)$-matrix Riesz sequences localize to $(n,d)$-matrix Gabor Riesz sequences for the localization, and $\tr_A (1_A) = 1$ and $\tr_B (1_B)= s(\La)$ (once again since the identity on $B$ is $s(\La)\delta_0$), the result is immediate by \cref{Theorem: Requirements for module Riesz}.
\end{proof}
\begin{rmk}
	 The two preceding propositions contain statements known as density theorems in Gabor analysis. This is due to the fact that they give conditions on the density of a lattice for there to possibly exist Gabor frames.
\end{rmk}
Now let $\La \subset G \times \widehat{G}$ be cocompact again. In this framework two of the cornerstones of Gabor analysis, namely the Wexler-Raz biorthogonality relations and the duality principle for Gabor frames, are then quite easy to prove for $(n,d)$-matrix Gabor frames for $L^2 (G)$ with respect to $\La$ with atoms in $M_{n,d}(E)$. 
\begin{prop}[Wexler-Raz Biorthogonality Relations]
	\label{Proposition: Wexler-Raz for (n,d)-matrix frames}
	Let $\La \subset G \times \widehat{G}$ be a closed and cocompact subgroup, and let $g, h \in M_{n,d}(E)$. Then the following are equivalent:
	\begin{enumerate}
		\item [i)] $\GS (g;\La)$ and $\GS(h;\La)$ are dual $(n,d)$-matrix Gabor frames for $L^2 (G)$.
		\item [ii)] $ \langle g, \pi (\la^{\circ}) h \rangle_{\ell^2 (\Lao \times \Z_n \times \Z_d)} = s(\La) \cdot \delta_{0, \la^{\circ}}$. 
	\end{enumerate}
\end{prop}
\begin{proof}
	As $\La$ is cocompact we know $\La^{\circ}$ is discrete, so $M_d (B)$ is unital. Knowing this, we can see that both the above statements are equivalent to the statement $[ g,h\bracketb = [h,g\bracketb = 1_{M_d (B)}$. 
\end{proof}

\begin{thm}[Duality principle] 
	\label{Theorem: Duality for (n,d)-matrix frames}
	Let $\La \subset \tfp{G}$ be a closed and cocompact subgroup, and let $g \in M_{n,d}(E)$. Then the following are equivalent.
	\begin{enumerate}
		\item [i)] $\GS (g;\La)$ is an $(n,d)$-matrix Gabor frame for $L^2 (G)$.  %with canonical dual window in $M_{n,d}(E)$.
		\item [ii)] $\GS(g;\Lao)$ is a $(d,n)$-matrix Gabor Riesz sequence for $L^2 (G)$.% with canonical dual window in $M_{n,d}(E)$. 
	\end{enumerate}
\end{thm}
\begin{proof}
	Statement i) can be seen to be equivalent to $[g,g\bracketb$ being invertible by \cref{Proposition: Module frame iff Gabor frame}. But statement ii) is also equivalent to $[g,g\bracketb$ being invertible by \cref{Proposition: Module Riesz sequence iff Gabor Riesz sequence} since we consider $\GS(g;\Lao)$, that is, we work over $M_d (B)$. This finishes the proof.
	%With the assumption on $\La$ we know the canonical dual $h:= S_g^{-1}g \in M_{n,d}(E)$, and statement (i) is equivalent to $[g,h\bracketb = 1_{M_d (B)}$. But the canonical biorthogonal system of $\GS(g;\Lao)$ is $\GS (h;\Lao)$ for $h = g [g,g\bracketb^{-1} = S_g^{-1}g$ , so we see by \cref{Dual Gabor Riesz sequences iff extend to identity} that statement (ii) is also equivalent to $[g,h\bracketb = 1_{M_d (B)}$. This finishes the proof.
\end{proof}

For completeness we also include the following result related to the duality principle. This is a strengthening of the corresponding result in \cite{jalu18duality}.
\begin{prop}
	Let $\La \subset \tfp{G}$ be closed and cocompact, and let $g,h \in M_{n,d}(E)$ be such that $[g,h\bracketb$ extends to the identity operator on $L^2 (G \times \Z_n \times \Z_d)$. Then $\bbracket g,h ]$ is an idempotent operator from $L^2 (G \times \Z_n \times \Z_d)$ onto $\overline{\linspan} \{ \bigoplus_{i \in \Z_n} \bigoplus_{j \in \Z_d} \pi(\lao) g_{i,j}\}$.
\end{prop}
\begin{proof}
	Since $[g,h\bracketb$ extends to the identity operator, %we deduce as before that $M_d (B)$ is unital and that 
	we have $[g,h\bracketb = [h,g\bracketb = 1_{M_d (B)}$. That $\bbracket g,h]$ is an idempotent then follows by \cref{Proposition: Identity induce idempotent}.
	%\begin{equation*}
	%\bbracket g,h ] \bbracket g,h ] = \bbracket \bbracket g,h] g, h] = \bbracket g [h,g\bracketb , h ] = \bbracket g \cdot 1_{M_d (B)} , h ] = \bbracket g,h].
	%\end{equation*}
	By \cref{Proposition: $E$ to $hB$ duality} we get that $\bbracket g,h ]$ is an idempotent from $M_{n,d} (E)$ onto $\overline{g M_d (B)}$. But this passes to the localization, and the localization of $\overline{g M_d (B)}$ is $$\overline{\linspan} \{ \bigoplus_{i \in \Z_n} \bigoplus_{j \in \Z_d} \pi(\lao) g_{i,j}\}\subset L^2 (G \times \Z_n \times \Z_d).$$
\end{proof}
Lastly in this section, we prove that whenever $\La$ is cocompact, there is a close relationship between the module frame bounds and the Gabor frame bounds in the localization. 
\begin{prop}
	\label{Proposition: Modular frame and Gabor frame same bounds}
	Let $\La \subset G \times \widehat{G}$ be a closed and cocompact subgroup. Then $g \in M_{n,d}(E)$ generates a module $(n,d)$-matrix frame for $E$ as an $A$-module with lower frame bound $C$ and upper frame bound $D$ if and only if $\GS (g;\La)$ is an $(n,d)$-matrix Gabor frame for $L^2 (G)$ with lower frame bound $C$ and upper frame bound $D$.
\end{prop}
\begin{proof}
	By \cref{Lemma: Optimal frame bounds for modular frame} it suffices to prove that the optimal frame bounds are equal for both the module frame and the Gabor frame. We know that the localization of a module $(n,d)$-matrix frame for $E$ as an $A$-module becomes an $(n,d)$-matrix Gabor frame for $L^2 (G)$ with respect to $\La$. Since $\La$ is cocompact, we also know that if $g \in M_{n,d}(E)$ is such that $\GS(g;\La)$ is an $(n,d)$-matrix Gabor frame for $L^2(G)$, then the canonical dual $S_g^{-1} g \in M_{n,d}(E)$ also. By \cref{Proposition: Module frame iff Gabor frame} we have $\rho (\modft_g) = S_g$. From standard Hilbert space frame theory we know that the optimal upper frame bound for $S_g$ is $\Vert S_g \Vert$, and the optimal lower frame bound for $S_g$ is $\Vert S_g^{-1} \Vert^{-1}$, see for example Section 5.1 of \cite{gr01}. We know by \cref{Localization norm preserved} that $\Vert \modft_g \Vert = \Vert \rho (\modft_g)\Vert = \Vert S_g \Vert$ and $\Vert \modft_g^{-1} \Vert = \Vert \rho (\modft_g^{-1})\Vert = \Vert S_g^{-1} \Vert$. The result then follows by \cref{Lemma: Optimal frame bounds for modular frame}.
\end{proof}
\begin{rmk}
	\label{Remark: Norm of frame operator with inner products}
	A straightforward calculation will show that $\Vert \modft_g \Vert = \Vert \bbracket g,g ] \Vert$. Indeed, as $\modft_g x = x [ g,g\bracketb$ and $[ g,g\bracketb$ is positive, we see that
	\begin{equation*}
	\Vert \modft_g \Vert = \sup_{\Vert f \Vert=1} \{ \Vert \bbracket f [ g,g\bracketb, f] \Vert \}, 
	\end{equation*}
	and it follows immediately that $$\Vert \modft_g \Vert \leq \Vert [ g,g \bracketb \Vert = \Vert \bbracket g,g ] \Vert.$$ Inserting $f = \Vert \bbracket g,g ] \Vert^{-1/2} g $ we obtain the equality. Note that this is the same upper bound we obtained in \cref{Proposition: Dual atom in E iff in H_E}.
\end{rmk}
\begin{corl}
	Let $g \in M_{n,d}(E)$ and let $\La \subset \tfp{G}$ be a closed and cocompact subgroup. Then $g$ has the same Bessel bound both as an $(n,d)$-matrix Gabor atom for $L^2(G)$ with respect to $\La$ and as a $(d,n)$-matrix Gabor atom for $L^2(G)$ with respect to $\Lao$.
\end{corl}
\begin{proof} 
	It suffices to prove that the optimal Bessel bounds agree. Let $D_{\La}$ be the optimal Bessel bound for $\GS (g;\La)$, and let $D_{\Lao}$ be the optimal Bessel bound for $\GS(g;\Lao)$. By \cref{Proposition: Modular frame and Gabor frame same bounds} and \cref{Remark: Norm of frame operator with inner products} it follows that $D_{\La} = \Vert S_g \Vert = \Vert \bbracket g,g ] \Vert$. But the analogous argument works with $\Lao$ instead of $\La$, since the important part for the setup with localization as done in this paper is that $\La$ or $\Lao$ is cocompact. Indeed, this is really a consequence of \cref{Localization norm preserved}. %Hence it follows that $D_{\Lao} = \Vert \langle g,g \brangle \Vert$. Since $M_{n,d}(E)$ is a Morita equivalence bimodule, we know that $\Vert \blangle g,g \rangle \Vert = \Vert \langle g,g \brangle \Vert$, and hence $D_{\La}= D_{\Lao}$. 
	Hence we may just as well apply the frame operator $S_g^B$, which is given by the continuous extension of multiplication by $\bbracket g,g]$ on the left. That is, for all $f \in M_{n,d}(E)$ we have
	\begin{equation*}
		S_g^B f = \bbracket g,g] f.
	\end{equation*}
	An analogous argument to the one in \cref{Remark: Norm of frame operator with inner products} will show that $D_{\Lao} = \Vert \bbracket g,g] \Vert$, %. Since $M_{n,d}(E)$ is a Morita equivalence bimodule, we know that $\Vert \blangle g,g \rangle \Vert = \Vert \langle g,g \brangle \Vert$, 
	and hence $D_{\La}= D_{\Lao}$. 
\end{proof}

\section*{Acknowledgement}
The authors wish to thank Ulrik Enstad for his valuable input and for finding an error in an earlier draft of the manuscript. We would also like to thank the Erwin Schr\"odinger Institute for their hospitality and support since part of this research was conducted while the authors attended the program ''Bivariant K-theory in Geometry and Physics".
\bibliographystyle{plain}
\bibliography{CTFA}

\end{document}